\documentclass{statsoc}
\usepackage{array,booktabs}
\usepackage{amsmath}
\usepackage{amssymb}
\usepackage{bbm}
\RequirePackage[colorlinks,citecolor=cyan,urlcolor=blue,linkcolor=blue]{hyperref}
\usepackage[a4paper]{geometry}                
\usepackage{wrapfigure}
\usepackage{subfig,float}

\usepackage{enumerate}
\usepackage{graphicx}
\usepackage{verbatim,epsfig,color,lscape}
\usepackage{epstopdf}
\usepackage{color}
\usepackage{natbib}
\def\R{\mathbb{R}}

\def\l{\left}
\def\r{\right}
\newcommand{\m}{\mathcal}
\newcommand{\mb}{\mathbb}
\newcommand\argmin{\mathop{\mbox{argmin}}}

\newcommand{\tr}{\mbox{tr}\,}
\newcommand{\var}{\mbox{Var}}

\newcommand{\eps}{\varepsilon}
\newcommand{\mf}[1]{\mathbf{#1}}

\newcommand{\card}{\mathrm{card}}

\newcommand{\med}[1]{\mbox{med}\left(#1\right)}
\newcommand{\medg}[1]{\mbox{med}_g\left(#1\right)}
\newcommand{\pr}[1]{\mathbb{P}{\left(#1\right)}}
\newcommand{\dotp}[2]{\left\langle#1,#2\right\rangle}

\newtheorem{assumption}{Assumption}

\newtheorem{lemma}{Lemma}
\newtheorem{theorem}{Theorem}
\newtheorem{remark}{Remark}
\newtheorem{fact}{Fact}
\newtheorem{corollary}{Corollary}

\title[Distributed Statistical Estimation]{Distributed Statistical Estimation and Rates of Convergence in Normal Approximation}
\author[S. Minsker and N. Strawn]{Stanislav Minsker \footnote{S. Minsker was partially supported by the National Science Foundation grant DMS-1712956.}}
\address{ }
\email{minsker@usc.edu}
\author[S. Minsker and N. Strawn]{Nate Strawn}
\address{ }
\email{nate.strawn@georgetown.edu}

\begin{document}

\begin{abstract}
This paper presents a class of new algorithms for distributed statistical estimation that exploit divide-and-conquer approach.
We show that one of the key benefits of the divide-and-conquer strategy is robustness, an important characteristic for large distributed systems. 
We establish connections between performance of these distributed algorithms and the rates of convergence in normal approximation, and prove non-asymptotic deviations guarantees, as well as limit theorems, for the resulting estimators. 
Our techniques are illustrated through several examples: in particular, we obtain new results for the median-of-means estimator, as well as provide performance guarantees for distributed maximum likelihood estimation. 

\end{abstract}

\section{Introduction.}
\label{sec:intro}

According to \citep{IBM}, ``Every day, we create 2.5 quintillion bytes of data Ñ so much that 90\% of the data in the world today has been created in the last two years alone. This data comes from everywhere: sensors used to gather climate information, posts to social media sites, digital pictures and videos.. to name a few. This data is big data''. 
Novel scalable and robust algorithms are required to successfully address the challenges posed by big data problems. 
This paper develops and analyzes techniques that exhibit \emph{scalability}, a necessary characteristic of modern methods designed to perform statistical analysis of large datasets, as well as \emph{robustness} that guarantees stable performance of distributed systems when some of the nodes exhibit abnormal behavior.

The computational power of a single computer is often insufficient to store and process modern data sets, and instead data is stored and analyzed in a distributed way by a cluster consisting of several machines. 
We consider a distributed estimation framework wherein data is assumed to be randomly assigned to computational nodes that produce intermediate results. 
We assume that no communication between the nodes is allowed at this first stage. 
On the second stage, these intermediate results are used to compute some statistic on the whole dataset; see figure \ref{fig:dnc} for a graphical illustration.
\begin{figure}[h]
  \centering
  \includegraphics[scale=0.09]{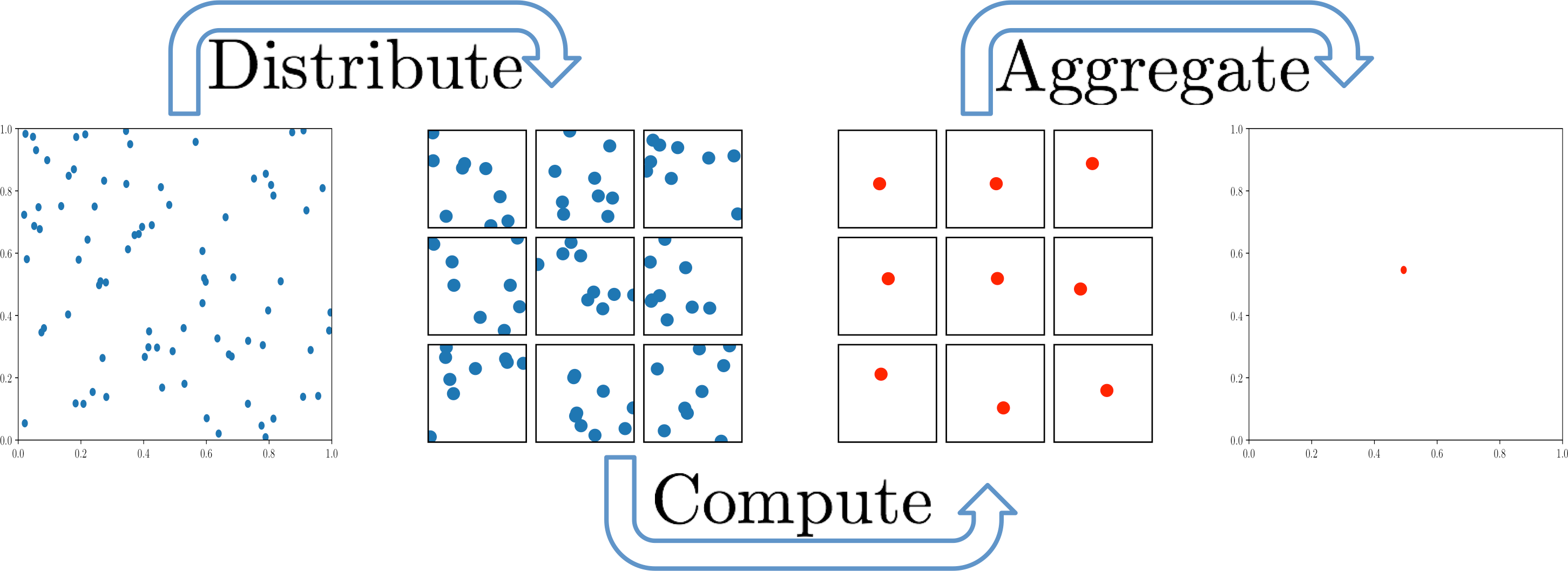}
\caption{Distributed estimation protocol where data is randomly distributed across nodes to obtain ``local'' estimates that are aggregated to compute a ``global'' estimate.}
\label{fig:dnc}
\end{figure}
Often, such a distributed setting is unavoidable in applications, whence interactions between subsamples stored on different machines are inevitably lost.
Most previous research focused on the following question: how significantly does this loss affect the quality of statistical estimation when compared to an ``oracle'' that has access to the whole sample? 
The question that we ask in this paper is different: what can be gained from randomly splitting the data across several subsamples? What are the statistical advantages of the divide-and-conquer framework? 
Our work indicates that one of the key benefits of an appropriate merging strategy is robustness. 
In particular, the quality of estimation attained by the distributed estimation algorithm is preserved even if a subset of machines stops working properly. 
At the same time, the resulting estimators admit tight probabilistic guarantees (expressed in the form of exponential concentration inequalities) even when the distribution of the data has heavy tails -- a viable model of real-world samples contaminated by outliers. 

We establish connections between a class of randomized divide-and-conquer strategies and the rates of convergence in normal approximation. 
Using these connections, we provide a new analysis of the ``median-of-means'' estimator which often yields significant improvements over the previously available results. 
We further illustrate the implications of our results by constructing novel algorithms for distributed Maximum Likelihood Estimation that admit strong performance guarantees under weak assumptions on the underlying distribution. 

\subsection{Background and related work.}
\label{sec:background}

We begin by introducing a simple model for distributed statistical estimation. 
Let $X_1,\ldots,X_N$ be a sequence of independent random variables with values in a measurable space $(S, \m S)$ that represent the data available to a statistician. 
We will assume that $N$ is large, and that that the sample $\m X=(X_1,\ldots,X_N)$ is partitioned into $k$ disjoint subsets $G_1,\ldots,G_k$ of cardinalities $n_j:=\card(G_j)$ respectively, where the partitioning scheme is independent of the data. 
Let $P_j$ be the distribution of $X_j$, $j=1,\ldots,N$. 
The goal is to estimate an unknown parameter 
$\theta_\ast=\theta_\ast(P_j), \ j=1,\ldots,N$ shared by $P_1,\ldots,P_N$ and taking values in a separable Hilbert space $(\mb H,\|\cdot\|_\mb H)$; for example, if $S=\mb H$, $\theta_\ast$ could be the common mean of $X_1,\ldots,X_N$. 
Distributed estimation protocol proceeds via performing ``local'' computations on each subset $G_j, \ j\leq k$, and the local estimators $\bar\theta_j:=\bar\theta_j(G_j), \ j\leq k$ are then pieced together to produce the final ``global'' estimator 
$\hat\theta^{(k)}= \hat\theta^{(k)}(\bar\theta_1,\ldots,\bar\theta_k)$. 
We are interested in the statistical properties of such distributed estimation protocols, and our main focus is on the final step that combines the local estimators. 
Let us mention that the condition requiring the sets $G_j, \ 1\leq j\leq k$ to be disjoint can be relaxed; we discuss the extensions related to U-quantiles in section \ref{sec:u-stat} below. 

The problem of distributed and communication~-~efficient statistical estimation has recently received significant attention from the research community. 
While our review provides only a subsample of the abundant literature in this field, it is important to acknowledge the works by \citet{mcdonald2009efficient,zhang2012communication,fan2014challenges,battey2015distributed,duchi2014optimality,shafieezadeh2015distributionally,lee2015communication,cheng2015computational,rosenblatt2016optimality,zinkevich2010parallelized}. 
\citet{li2016simple,scott2016bayes,shang2015bayesian,minsker2014robust} have investigated closely related problems for distributed Bayesian inference. 
Applications to important algorithms such as Principal Component Analysis were investigated in \citep{fan2017distributed,liang2014improved}, among others.  
\citet{jordan2013statistics}, author provides an overview of recent trends in the intersection of the statistics and computer science communities, describes popular existing strategies such as the ``bag of little bootstraps'', as wells as successful applications of the divide-and-conquer paradigm to problems such as matrix factorization. 

The majority of the aforementioned works propose \emph{averaging} of local estimators as a final merging step. 
Indeed, averaging reduces variance, hence, if the bias of each local estimator is sufficiently small, their average often attains optimal rates of convergence to the unknown parameter $\theta_\ast$. 
For example, when $\theta_\ast(P)=\mb E_P X$ is the mean of $X$ and $\bar\theta_j$ is the sample mean evaluated over the subsample $G_j, \ j=1,\ldots,k$, then the average of local estimators $\tilde\theta= \frac{1}{k}\sum_{j=1}^k \bar\theta_j$ is just a empirical mean evaluated over the whole sample. 
More generally, it has been shown by \citet{battey2015distributed,zhang2013divide} that in many problems (for instance, linear regression), $k$ can be taken as large as $O(\sqrt{N})$ without negatively affecting the estimation rates; similar guarantees hold for a variety of M-estimators \citep[see][]{rosenblatt2016optimality}. 
However, if the number of nodes $k$ itself is large (the case we are mainly interested in), then the averaging scheme has a drawback: if one or more among the local estimators $\bar\theta_j$'s is anomalous (for example, due to data corruption or a computer system malfunctioning), then statistical properties of the average will be negatively affected as well. 
For large distributed systems, this drawback can be costly. 

One way to address this issue is to replace averaging by a more robust procedure, such as the median or a robust M-estimator; this approach is investigated in the present work. 
In the univariate case ($\theta_\ast\in \mb R)$, the merging strategies we study can be described as solutions of the optimization problem
\begin{align}
\label{eq:huber1}
\widehat\theta^{(k)} = \argmin_{z\in \mb R} \sum_{j=1}^k \rho\l( |\bar\theta_j - z| \r)
\end{align}
for an appropriately defined convex function $\rho$; we investigate this class of estimators in detail.  
A natural extension to the case $\theta_\ast\in \mb R^m$ is to consider 
\[
\widehat\theta^{(k)} = \argmin_{y\in \mb R^m} \sum_{j=1}^k \rho\l(  \l\| \bar\theta_j - y \r\|_\circ\r)
\] 
for some convex function $\rho$ and norm $\|\cdot\|_\circ$. For example, if $\rho(x)=x$, then $\widehat\theta^{(k)}$ becomes the spatial (also known as geometric or Haldane's) median \citep{haldane1948note,small1990survey} of $\bar\theta_1,\ldots,\bar\theta_k$. Since the median remains stable as long as at least a half of the nodes in the system perform as expected, such model for distributed estimation is robust.
The merging approach based on the various notions of the multivariate median has been previously considered by \citet{minsker2015geometric} and \citet{hsu2016loss}; here, we analyze the setting when $\rho(x)=x$ and $\|\cdot\|_\circ$ is the $L_1$-norm using the novel approach.

Existing results for the median-based merging strategies have several pitfalls related to the deviation rates, and in most cases known guarantees are suboptimal. 
In particular, these guarantees suggest that estimators obtained via the median-based approach are very sensitive to the choice of $k$, the number of partitions. 
For instance, consider the problem of univariate mean estimation, where $X_1,\ldots,X_N$ are i.i.d. copies of $X\in \mb R$, and $\theta_\ast = \mb E X$ is the expectation of $X$. 
Assume that $\card(G_j)\geq n:=\lfloor N/k \rfloor$ for all $j$, let $\bar\theta_j = \frac{1}{|G_j|}\sum_{i: X_i\in G_j} X_i$ be the empirical mean evaluated over the subsample $G_j$, and define the ``median-of-means'' estimator via
\begin{align}
\label{eq:med-of-means}
&\widehat\theta^{(k)}
=\med{\bar\theta_1,\ldots,\bar\theta_k},
\end{align}
where $\med{\cdot}$ is the usual univariate median. 
This estimator has been introduced by \citet{Nemirovski1983Problem-complex00} in the context of stochastic optimization, and later appeared in \citep{jerrum1986random} and \citep{alon1996space}. 
If $\var(X)=\sigma^2<\infty,$ it has been shown \citep[for example, by][]{lerasle2011robust} that the median-of-means estimator 
$\widehat\theta^{(k)}$ satisfies
\begin{align}
& \label{eq:med0}
\l| \widehat\theta^{(k)} - \theta_\ast \r|\leq 2\sigma\sqrt{6e} \sqrt{\frac{k}{N}}
\end{align}
with probability $\geq 1-e^{-k}$. 
However, this bound, while being the current state of the art, does not tell us what happens at the confidence levels other than $1-e^{-k}$. 
For example, if $k=\lfloor \sqrt{N} \rfloor$, the only conclusion we can make is that $\l| \widehat\theta^{(k)} - \theta_\ast \r| \lesssim N^{-1/4}$ with high probability, which is far from the optimal rate $N^{-1/2}$. 
And if we want the bound to hold with confidence 99\% instead of $1 - e^{-\sqrt N}$, then, according to \eqref{eq:med0}, we should take $k=\lfloor \log 100 \rfloor +1 = 5$, in which case the beneficial effect of parallel computation is very limited. 
The natural question to ask is the following: is the median-based merging step indeed suboptimal for large values of $k$ (e.g., $k=\lfloor \sqrt{N} \rfloor$), or is the problem related to the suboptimality of existing bounds? 
We claim that in many situations the latter is the case, and that previously known results can be strengthened: for instance, the statement of Corollary \ref{corollary:med-of-means1} below implies that whenever $\mb E| X - \theta_\ast |^3<\infty$, the median-of-means estimator satisfies 
\[
| \widehat\theta^{(k)} - \theta_\ast | \leq
3\sigma \l(\frac{\mb E \l| X-\theta_\ast \r|^3}{\sigma^3}\frac{k}{N-k} + \sqrt{\frac{s}{N - k}} \r)
\]
with probability $\geq 1 - 4 e^{-2s}$, for \emph{all} $s \lesssim k$. 
In particular, this inequality shows that the estimator \eqref{eq:med-of-means} has ``typical'' deviations of order $N^{-1/2}$ whenever $k=O(\sqrt N)$, hence the ``statistical cost'' of employing a large number of computational nodes is minor. Moreover, we will prove that $\sqrt{N}\l( \widehat \theta^{(k)} - \theta_\ast \r) \xrightarrow{d} N\l(0, \frac{\pi}{2}\sigma^2 \r)$ if $k\to\infty$ and $k=o(\sqrt{N})$ as $N\to \infty$. 
It will also be demonstrated that improved bounds hold in other important scenarios, such as maximum likelihood estimation, even when the subgroups have different sizes and the observations are not identically distributed. 

\subsection{Organization of the paper. }

Section \ref{sec:notation} describes notation used throughout the paper. 
Sections \ref{ssec:univar} and \ref{ssec:multivar} present main results and examples for the cases of univariate and multivariate parameter respectively. 
Outcomes of numerical simulation are discussed in section \ref{ssec:simul}, and proofs of the main results are contained in section \ref{sec:proofs}.

\subsection{Notation.}
\label{sec:notation}

Everywhere below, $\|\cdot\|_1$ and $\|\cdot\|_2$ stand for the $L_1$ and $L_2$ norms of a vector, and $\|\cdot\|$ - for the operator norm of a matrix (its largest singular value). 

Given a probability measure $P$, $\mb E_P(\cdot)$ will stand for the expectation with respect to $P$, and we will write $\mb E(\cdot)$ when $P$ is clear from the context. 
Convergence in distribution will be denoted by $\xrightarrow{d}$. 

For two sequences $\l\{ a_j\r\}_{j\geq 1}\subset \mb R$ and $\l\{ b_j \r\}_{j\geq 1}\subset \mb R$ for $j\in\mb N$, the expression 
$a_j\lesssim b_j$ means that there exists a constant $c>0$ such that $a_j\leq c b_j$ for all $j\in \mb N$. 
Absolute constants will be denoted $c,C,c_1$, etc., and may take different values in different parts of the paper. 
For a function $f:\R^d\mapsto\R$, we define 
\[
\argmin_{z\in\R^d} f(z) = \{z\in\R^d: f(z)\leq f(x)\text{ for all }x\in\R^d\},
\]
and $\|f\|_\infty:=\mathrm{ess \,sup}\{ |f(x)|: \, x\in \R^d\}$. 
Finally, $f_+(x) =  \lim_{t\searrow 0} \frac{f(x+t) - f(x)}{t}$ and $f_-(x) = \lim_{t\nearrow 0} \frac{f(x+t) - f(x)}{t}$ will denote the right and left derivatives of $f$ respectively (whenever these limits exist).  
Additional notation and auxiliary results are introduced on demand for the proofs in section \ref{sec:proofs}.

\subsection{Main results.}	
\label{sec:main}

As we have argued above, existing guarantees for the estimator \eqref{eq:med-of-means} are sensitive to the choice of $k$, the number of partitions. 
In the following sections, we demonstrate that these bounds are often suboptimal, and show that large values of $k$ often do not have a significant negative effect on the statistical performance of resulting algorithms. 
 
The key observation underlying the subsequent exposition is the following: assume that the ``local estimators'' $\bar \theta_j, \ 1\leq j\leq k$, are asymptotically normal with asymptotic mean equal to $\theta_\ast$. 
In particular, distributions of $\bar\theta_j$'s are approximately symmetric, with $\theta_\ast$ being the center of symmetry. 
The location parameters of symmetric distributions admits many robust estimators of the form \eqref{eq:huber1}, the sample median being a notable example. 

This intuition allows us to establish a parallel between the non-asymptotic deviation guarantees for distributed estimation procedures of the form \eqref{eq:huber1} and the degree of symmetry of ``local'' estimators quantified by the rates of convergence to normal approximation. 
Results for the univariate case are presented in section \ref{ssec:univar}, and extensions to the multivariate case are presented in section \ref{ssec:multivar}.

\section{The univariate case.}			
\label{ssec:univar}	

We assume that $X_1,\ldots,X_N$ is a collection of independent (but not necessarily identically distributed) $S$-valued random variables with distributions $P_1,\ldots,P_N$ respectively. 
The data are partitioned into disjoint groups $G_1,\ldots,G_k$ of cardinality $n_j: = \card(G_j)$ each, and such that $\sum_{j=1}^k n_j = N$. 
Let $\bar\theta_j:=\bar\theta_j(G_j), \ 1\leq j\leq k$ be a sequence of independent estimators of the parameter $\theta_\ast\in \mb R$ shared by $P_1,\ldots,P_N$. 
Our main assumption will be that $\bar\theta_1,\ldots,\bar\theta_k$ are asymptotically normal as quantified by the following condition. 
\begin{assumption}
\label{ass:1}
Let $\Phi(t)$ be the cumulative distribution function of the standard normal random variable $Z\sim N(0,1)$. 
For each $j=1,\ldots,k$, there exist a sequence $\{\sigma^{(j)}_n\}_{n\in \mb N}\subset \mb R_+$ such that 
\[
g_j(n_j):=\sup_{t\in \mb R}\l| \pr{ \frac{\bar\theta_j - \theta_\ast}{\sigma^{(j)}_{n_j}}\leq t} - \Phi( t ) \r| \to 0 \text{ as }n_j\to\infty.
\]
\end{assumption}
Clearly, functions $g_j(n_j)$, control the \emph{rate of convergence} of estimators $\bar\theta_1,\ldots,\bar\theta_k$ to the normal law. 
Furthermore, let 
\[
H_k:=\l( \frac{1}{k}\sum_{j=1}^k \frac{1}{\sigma_{n_j}^{(j)}} \r)^{-1}
\] 
be the \emph{harmonic mean} of $\sigma_{n_j}^{(j)}$'s, and set 
$\alpha_j=\frac{ H_k }{\sigma_{n_j}^{(j)}}$. 
Note that $\sum_{j=1}^k \alpha_j = k$, and that $\alpha_1=\ldots=\alpha_k=1$ if $\sigma_{n_1}^{(1)}=\ldots=\sigma_{n_k}^{(k)}$. 
 
\subsection{Merging procedure based on the median.}
\label{ssec:median}

In this subsection, we establish guarantees for the merging procedure based on the sample median, namely, 
\begin{align*}
&
\widehat\theta^{(k)}=\med{\bar\theta_1,\ldots,\bar\theta_k}.
\end{align*}
This case is treated separately due to its practical importance, the fact that we can obtain better numerical constants, and a conceptually simpler proof. 
\begin{theorem}
\label{th:main1}
Assume that $s>0$ and $n_j = \card(G_j), \ j=1,\ldots, k$ are such that 
\begin{align}
\label{eq:00}
&
\frac{1}{k}\sum_{i=1}^k \l( g_i(n_i) + \sqrt{\frac{s}{k}} \r) \cdot \max_{j=1,\ldots,k} \alpha_j < \frac{1}{2}.
\end{align} 
Moreover, let assumption \ref{ass:1} be satisfied, and let $\zeta_j(n_j,s)$ solve the equation 
\[
\Phi \l( \zeta_j(n_j,s)/\sigma_{n_j}^{(j)} \r) - \frac{1}{2}=\alpha_j \cdot \frac{1}{k}\sum_{i=1}^k \l( g_i(n_i) + \sqrt{\frac{s}{k}} \r).
\] 
Then for all $s$ satisfying \eqref{eq:00},
\[
\l| \widehat\theta^{(k)} - \theta_\ast \r| \leq \zeta(s):=\max_{j=1,\ldots,k} \zeta_j(n_j,s)
\]
with probability at least $1 - 4e^{-2s}$.
\end{theorem}

\begin{proof}
See section \ref{sec:proof1}.
\end{proof}
The following lemma yields a more explicit form of the bound and numerical constants.
\begin{lemma}
\label{cor:0}
Assume that $\frac{1}{k}\sum_{i=1}^k \l( g_i(n_i) + \sqrt{\frac{s}{k}} \r)\cdot \max\limits_{j=1,\ldots,k} \alpha_j \leq 0.33$. Then 
\begin{align*}
&
\zeta(s)\leq 3H_k \cdot \frac{1}{k}\sum_{j=1}^k \l( g_j(n_j) + \sqrt{\frac{s}{k}} \r).
\end{align*}
\end{lemma}
\begin{proof}
See section \ref{proof:cor0}.
\end{proof}
\begin{remark}
Let $\bar\sigma^{(1)}\leq \ldots \leq \bar\sigma^{(k)}$ be the non-decreasing rearrangement of $\sigma_{n_1}^{(1)},\ldots,\sigma_{n_k}^{(k)}$. 
It is easy to see that the harmonic mean $H_k$ of $\sigma_{n_1}^{(1)},\ldots,\sigma_{n_k}^{(k)}$ satisfies 
\[
H_k \leq \frac{k}{\lfloor k/m \rfloor} \cdot \frac{1}{\lfloor k/m \rfloor}\sum_{j=1}^{\lfloor k/m \rfloor} \bar\sigma^{(j)}
\]
for any integer $1\leq m\leq k$, hence, informally speaking, the deviations of $\widehat \theta^{(k)}$ are controlled by the smallest $\sigma_{n_j}^{(j)}$'s rather than the largest. 
\end{remark}
 
\subsection{Example: new bounds for the median-of-means estimator.}
\label{ssec:med-of-means}

The univariate mean estimation problem is pervasive in statistics, and serves as a building block of more advanced methods such as empirical risk minimization. 
Early works on robust mean estimation include Tukey's ``trimmed mean'' \citep{tukey1946sampling}, as well as ``winsorized mean''  \citep{bickel1965some}; also see discussion in \citep{bubeck2013bandits}. 
These techniques often produce estimators with significant bias. 
A different approach based on M-estimation was suggested by O. Catoni \citep{catoni2012challenging}; Catoni's estimator yields almost optimal constants, however, its construction requires additional information about the variance or the kurtosis of the underlying distribution; moreover, its computation is not easily parallelizable, therefore this technique cannot be easily employed in the distributed setting. 

Here, we will focus on a fruitful idea that is commonly referred to as the ``median-of-means'' estimator that was formally defined in equation \eqref{eq:med-of-means} above. 
Several refinements and extensions of this estimator to higher dimensions have been recently introduced by \citet{minsker2015geometric,Hsu2013Loss-minimizati00,devroye2016sub,joly2016estimation,lugosi2017sub}. 
Advantages of this method include the facts that that it can be implemented in parallel and does not require prior knowledge of any information about parameters of the distribution (e.g., its variance). 
The following result for the median-of-means estimator is the corollary of Theorem \ref{th:main1}; for brevity, we treat only the i.i.d. case. Recall that $n=\lfloor N/k \rfloor$ and $\card(G_j)\geq n, \ j=1,\ldots,k$.
 
\begin{corollary}
\label{corollary:med-of-means1}
Let $X_1,\ldots,X_N$ be a sequence of i.i.d. copies of a random variable $X\in \mb R$ such that $\mb EX = \theta_\ast$, $\var(X)=\sigma^2$, $\mb E| X-\theta_\ast |^3<\infty$, and set $c_n = 0.4748\frac{\mb E|X - \theta_\ast|^3}{\sigma^3 \sqrt{n}}$. 
Then for all $s>0$ such that $c_n + \sqrt{\frac{s}{k}}\leq 0.33$, the estimator $\widehat \theta^{(k)}$ defined in \eqref{eq:med-of-means} satisfies 
\[
| \widehat\theta^{(k)} - \theta_\ast | \leq
\sigma \l( 1.43 \frac{\mb E \l| X-\theta_\ast \r|^3/\sigma^3}{n} + 3\sqrt{\frac{s}{kn}} \r)
\]
with probability at least $1 - 4 e^{-2s}$.
\end{corollary}
\begin{remark}
The term $1.43\sigma \frac{\mb E \l| X-\theta_\ast \r|^3/\sigma^3}{n}$ can be thought of as the ``bias'' due to asymmetry of the distribution of the sample mean. Note that whenever $k\lesssim\sqrt{N}$ (so that $n\gtrsim \sqrt{N}$), the right-hand side of the inequality above is of order $(kn)^{-1/2}\simeq N^{-1/2}$. 
\end{remark}
\begin{proof}
It follows from the Berry-Esseen Theorem (fact \ref{th:BE} in section \ref{sec:prelim}) that assumption \ref{ass:1} is satisfied with $\sigma_n^{(1)}=\ldots=\sigma_n^{(k)}=\frac{\sigma}{\sqrt{n}}$, and
\[
g_j(n)\leq c_n = 0.4748\frac{\mb E|X - \theta_\ast|^3}{\sigma^3 \sqrt{n}}
\]
for all $j$. Lemma \ref{cor:0} implies that $\max_j\zeta_j(n,s) \leq 3\frac{\sigma}{\sqrt{n}}\l( c_n + \sqrt{s/k}\r)$, and the claim follows from Theorem \ref{th:main1}. 	
\end{proof}
For distributions with infinite third moment, the rate of convergence in the Berry-Esseen type bound is slower, and the following result holds instead. 
\begin{corollary}
\label{corollary:med-of-means1.5}
Let $X_1,\ldots,X_N$ be a sequence of i.i.d. copies of a random variable $X\in \mb R$ such that $\mb EX = \theta_\ast$, $\var(X)=\sigma^2$, $\mb E| X - \theta_\ast |^{2+\delta}<\infty$ for some $\delta\in(0,1]$. 
Then there exist absolute constants $c_1,c_2>0$ such that for all $s>0$ and $k$ satisfying 
$\frac{\mb E |X -\theta_\ast |^{2+\delta}}{\sigma^{2+\delta}n^{\delta/2}} + \sqrt{\frac{s}{k}}\leq c_1$, the following inequality holds with probability at least $1 - 4 e^{-2s}$:
\[
| \widehat\theta^{(k)} - \theta_\ast | \leq
c_2\sigma \l( \frac{\mb E \l| X-\theta_\ast \r|^{2+\delta}/\sigma^{2+\delta}}{n^{\frac{1+\delta}{2}}} + \sqrt{\frac{s}{N}} \r).
\]
\end{corollary}
In this case, typical deviations of $\widehat \theta^{(k)}$ are still of order $N^{-1/2}$ as long as $k\lesssim N^{\delta/(1+\delta)}$. The proof of this result follows from fact \ref{th:BE-gen} in section \ref{sec:prelim} in the same way as Corollary \ref{corollary:med-of-means1} was deduced from the Berry-Esseen bound.
 
\subsection{Example: distributed maximum likelihood estimation.}
\label{ssec:mle}

Let $X_1,\ldots,X_N$ be i.i.d. copies of a random vector $X\in \mb R^d$ with distribution $P_{\theta_\ast}$, where $\theta_\ast\in \Theta\subseteq \mb R$. 
Assume that for each $\theta\in \Theta$, $P_\theta$ is absolutely continuous with respect to a $\sigma$-finite measure $\mu$, and let $p_\theta=\frac{dP_\theta}{d\mu}$ be the corresponding density. 
In this section, we state sufficient conditions for assumption \ref{ass:1} to be satisfied when $\bar\theta_1,\ldots,\bar\theta_k$ are the maximum likelihood estimators \citep{Vaart1998Asymptotic-stat00} of $\theta_\ast$. 
Conditions stated below were obtained by \citet{pinelis2016optimal}. 
All derivatives below (denoted by $'$) are taken with respect to $\theta$, unless noted otherwise.

Assume that the the log-likelihood function $\ell_x(\theta)= \log p_\theta(x)$ satisfies the following:
\begin{enumerate}[(1)]
\item $[\theta_\ast - \delta, \theta_\ast + \delta] \subseteq \Theta$ for some $\delta>0$;
\item ``standard regularity conditions'' that allow differentiation under the expectation: assume that $\mb E\ell'_X(\theta_\ast)=0$, and that the Fisher information $\mb E\ell'_X(\theta_\ast)^2=-\mb E \ell''_X(\theta_\ast):=I(\theta_\ast)$ is finite;
\item $\mb E \l| \ell'_X(\theta_\ast) \r|^3 + \mb E \l| \ell''_X(\theta_\ast) \r|^3 <\infty$;
\item for $\mu$-almost all $x$, $\ell_x(\theta)$ is three times differentiable for $\theta\in[\theta_\ast-\delta,\theta_\ast+\delta]$, and 
$\mb E \sup_{|\theta-\theta_\ast|\leq \delta} \l| \ell_X'''(\theta) \r|^3<\infty$; 
\item $\pr{|\bar\theta_1 - \theta_\ast|\geq \delta}\leq c \gamma^n$ for some positive constants $c$ and $\gamma\in[0,1)$.
\end{enumerate}
In turn, condition (5) above is implied by the following two inequalities \citep[see][section 6.2, for detailed discussion and examples]{pinelis2016optimal}:
\begin{enumerate}
\item $H^2(\theta,\theta_\ast) \geq 2 - \frac{2}{\l( 1+ c_0(\theta - \theta_\ast)^2 \r)^\gamma}$, 
where $H(\theta_1,\theta_2)=\sqrt{\int_{\mb R^d} \l( \sqrt{p_{\theta_1}} - \sqrt{p_{\theta_2}}\r)^2 d\mu}$ is the Hellinger distance, and $c_0,\gamma$ are positive constants;
\item $I(\theta)\leq c_1 + c_2\l| \theta\r|^\alpha$ for some positive constants $c_1,c_2$ and $\alpha$ and all $\theta\in \Theta$. 
\end{enumerate}
\begin{corollary}
Assume that conditions (1)-(5) are satisfied, and that $\card(G_j)\geq n = \lfloor N/k \rfloor, \ j=1,\ldots,k$. 
Then for all $s>0$ such that $\frac{\mathfrak C}{\sqrt{n}} + c\gamma^n + \sqrt{\frac{s}{k}}\leq 0.33$,  
\[
\l| \widehat \theta^{(k)} - \theta_\ast \r|\leq 
\frac{3}{\sqrt{I(\theta_\ast)}}\l( \frac{\mathfrak C}{n} + \frac{c}{\sqrt{n}}\gamma^n + \sqrt{\frac{s}{kn}} \r)
\]
with probability at least $1- 4 e^{-2s}$, where $\mathfrak C$ is a positive constant that depends only on $\{P_\theta\}_{\theta\in[\theta_\ast-\delta,\theta_\ast+\delta]}$.
\end{corollary}
\begin{proof}
It follows from results in \citep{pinelis2016optimal}, in particular equation (5.5), that whenever conditions (1)-(5) hold, assumption \ref{ass:1} is satisfied for all $j$ with $\sigma_n^{(j)}=\l( n I(\theta_\ast) \r)^{-1/2}$, where $I(\theta_\ast)$ is the Fisher information, and $g_j(n) \leq \frac{\mathfrak C}{\sqrt{n}} + c\gamma^n$, where $\mathfrak C$ is a constant that depends only on $\{P_\theta\}_{\theta\in[\theta_\ast-\delta,\theta_\ast+\delta]}$. 
Lemma \ref{cor:0} implies that 
\[
\max_{j=1,\ldots,k} \zeta_j(n,s) \leq 3\l( \frac{\mathfrak C}{\sqrt{n}} +c\gamma^n + \sqrt{s/k}\r),
\] 
and the claim follows from Theorem \ref{th:main1}. 
\end{proof}
\begin{remark}
Results of this section can be extended to include other M-estimators besides MLEs, as \cite{bentkus1997berry} have shown that M-estimators satisfy a variant of Berry-Esseen bound under rather general conditions. 
\end{remark}

\subsection{Merging procedures based on robust M-estimators.}
\label{ssec:m-estimators}

In this subsection, we study the family of merging procedures based on the M-estimators
\begin{align}
\label{eq:M-est}
\widehat\theta^{(k)}_\rho := \argmin_{z\in \mb R}\sum_{j=1}^k \rho\l(z - \bar\theta_j\r).
\end{align}
The sample median $\med{\bar\theta_1,\ldots,\bar\theta_k}$ corresponds to the choice of (non-smooth) $\rho(x)=|x|$ and was treated separately above; here, it will be assumed that $\rho$ is convex, even, differentiable function such that $\rho(z)\to\infty$ as $|z|\to\infty$ and $\|\rho' \|_\infty<\infty$. 
A particular example of such a function is Huber's loss 
\begin{align}
\label{eq:huber}
\rho_M(z) = \begin{cases}
z^2/2, & |z|\leq M, \\
M |z|-M^2/2 , & |z|>M,
\end{cases}
\end{align}
where $M$ is a positive constant. 
The following result quantifies non-asymptotic performance of the estimator $\widehat\theta^{(k)}_\rho$. As before, we set
\begin{align}
\label{eq:alpha}
H_k &= \frac{1}{1/k\sum_{i=1}^k 1/\sigma_{n_j}^{(i)}} \text{ and }
\alpha_j =\frac{ H_k }{\sigma_{n_j}^{(j)}},
\end{align}
where $\sigma_n^{(j)}$'s are defined in assumption \ref{ass:1}. 
Moreover, given the loss $\rho$ as above, let $C_\rho>0$ be such that $|\rho'(x)|\geq \frac{\|\rho'\|_\infty}{2}$ for $|x| > C_\rho$. 
\begin{theorem}
\label{th:M-est}
Let assumption \ref{ass:1} be satisfied, and suppose that $s>0$ and $n_1,\ldots,n_k$ are such that 
\begin{align}
\label{eq:M-est-hypo}
&
\max_{j=1,\ldots,k} \alpha_j \,e^{\l( C_\rho/\sigma_{n_j}^{(j)}\r)^2} \frac{1}{k}\sum_{i=1}^k\l( \sqrt{\frac{s}{k}} + 2 g_i(n_i) \r) \leq 0.33.
\end{align} 
Then for all $s$ satisfying \eqref{eq:M-est-hypo},
\begin{align}
\label{eq:th2-bound}
&
\l| \widehat\theta^{(k)}_\rho - \theta_\ast \r| \leq 3 H_k\max_{j=1,\ldots,k} e^{\l( C_\rho/\sigma_{n_j}^{(j)}\r)^2} \cdot 
\frac{1}{k}\sum_{i=1}^k\l( \sqrt{\frac{s}{k}} + 2 g_i(n_i) \r)
\end{align}
with probability at least $1 - 4e^{-2s}$.
\end{theorem} 
\begin{proof}
See section \ref{sec:proof-M}. 
\end{proof}
Note that the bound depends on $\rho$ only through $\max_{j=1,\ldots,k} e^{\l( C_\rho/\sigma_{n_j}^{(j)}\r)^2}$. 
Assume for concreteness that $n_1=\ldots=n_k=\lfloor N/k \rfloor$, and that $\rho=\rho_M$ is Huber's loss defined in \eqref{eq:huber}, so that $C_\rho = M/2$.  
For $\max_{j=1,\ldots,k} e^{\l( C_\rho/\sigma_{n_j}^{(j)}\r)^2}$ to be bounded above by an absolute constant, one should choose 
$M$ to be of order $\min_{j=1,\ldots,k}\sigma_{n_j}^{(j)}$. 
While the latter quantity is typically unknown, it can be estimated in some cases. 
For example, if the data are i.i.d. then $\sigma_{n_j}^{(j)} = \sqrt{\var\l(\bar \theta_1 \r)}$ for all $j$. 
Since $\bar\theta_j$'s are approximately normal, their standard deviation can be estimated by the \emph{median absolute deviation} as 
\[
\widehat{\sigma}_{n_1} = \frac{1}{\Phi^{-1}(0.75)}\med{|\bar\theta_1 - \med{\bar\theta_1,\ldots,\bar\theta_k}|,\ldots,|\bar\theta_k - \med{\bar\theta_1,\ldots,\bar\theta_k}| },
\]
where the factor $1/\Phi^{-1}(0.75)$ is introduced to make the estimator consistent \citep{hampel2011robust}; another possibility is to use bootstrap \citep{ghosh1984note}.


\subsection{Asymptotic results.}
\label{ssec:asymptotic}

In this section, we complement the previously discussed non-asymptotic deviation bounds for $\widehat\theta^{(k)}_\rho$ by the asymptotic results. For the benefits of clarity, we state the complete list of assumptions made below:
\begin{enumerate}
\item $X_1,\ldots,X_N$ are i.i.d., $n=\lfloor N/k \rfloor$ and $\card(G_j) = n, \ j=1,\ldots,k$; result for non-identically distributed data is presented in Appendix \ref{appendix:clt}. 
\item Assumption \ref{ass:1} is satisfied for some function $g(n)$ (note that there is no dependence on index $j$ due to the i.i.d. assumption);
\item $k$ and $n$ are such that $k\to\infty$ and $\sqrt{k}\cdot g(n)\to 0$ as $N\to \infty$;
\item $\rho$ is a convex, even function, such that $\rho(z)\to\infty$ as $|z|\to\infty$ and $\|\rho' \|_\infty<\infty$ (here, $\rho'(x)$ is defined as the average of the right and left derivatives of $\rho$ at $x$).
\item $\widehat\theta^{(k)}_\rho$ is defined as 
\[
\widehat\theta^{(k)}_\rho := \argmin_{z\in \mb R}\sum_{j=1}^k \rho\l(\frac{z - \bar\theta_j}{\sigma_n}\r),
\]
where $\sigma_n^{(1)}=\ldots = \sigma_n^{(k)}\equiv \sigma_n$ is a normalizing sequence from assumption \ref{ass:1} (our  definition of the estimator is slightly different than in section \ref{ssec:m-estimators} which allows to keep $\rho$ fixed as $k$ and $n$ are changing). 
\end{enumerate}
For $z\in \mb R$, define 
\[
L(z):= \mb E \rho'\l(z+Z\r),
\] 
where $Z\sim N(0,1)$. Note that, since $\rho$ is differentiable almost everywhere, $L(z) = \mb E \rho'_-(z+Z) = \mb E \rho'_+(z+Z)$. 
\begin{theorem}
\label{th:clt}
Under assumptions (a)-(e) above, 
\[
\sqrt{k}\,\frac{\widehat\theta^{(k)}_\rho - \theta_\ast}{\sigma_n}\xrightarrow{d} N(0,\Delta^2),
\]
where $\Delta^2 = \frac{\mb E\l( \rho'(Z)\r)^2}{\l( L'(0) \r)^2}$.
\end{theorem}
\begin{proof}
See section \ref{sec:proofclt}.
\end{proof}
For example, if $\rho(x)=|x|$, Theorem \ref{th:clt} implies that under appropriate assumptions, the median-of-means estimator $\widehat \theta^{(k)}$ defined in \eqref{eq:med-of-means} satisfies 
\[
\sqrt{N}\l( \widehat \theta^{(k)} - \theta_\ast \r) \xrightarrow{d} N\l(0, \frac{\pi}{2}\sigma^2 \r).
\]
Indeed, in this case $\sigma_n = \sigma/\sqrt{n}$, where $\sigma^2 = \var(X_1)$, and 
\[
\rho'(x)=\begin{cases}
-1, & x<0, \\
0, & x=0, \\
1, & x>0,
\end{cases}
\]
hence a simple calculation yields $\Delta^2 = 1/(L'(0))^2 = \pi/2$. 

If we consider the mean estimation problem with Huber's loss $\rho_M(x)$ \eqref{eq:huber} instead of $\rho(x)=|x|$, we similarly deduce that 
\[
\rho'(x)=\begin{cases}
-M & x\leq -M, \\
x, & |x|<M, \\
M, & x\geq M,
\end{cases}
\]
and we get the well-known \citep{huber1964robust} expression $\Delta^2 = \frac{\int_{-M}^M x^2 d\Phi(x) + 2M^2(1 - \Phi(M))}{\l( 2\Phi(M) -1 \r)^2}$; in particular, $\Delta^2 \to 1$ as $M\to \infty$, and the convergence is fast. For instance, $\Delta^2 \simeq 1.15$ for $M=2$ and $\Delta^2 \simeq 1.01$ for $M=3$. 

\begin{remark}
The key assumptions in the list (a)-(e) governing the regime of growth of $k$ and $n$ are (b) and (c). For instance, if the random variables possess finite moments of order $(2+\delta)$ for some $\delta\in(0,1]$, then it follows from fact \ref{th:BE-gen} in section \ref{sec:prelim} that $\sqrt{k}\,g(n)\to 0 \text{ if } k=o\l(N^{\frac{\delta}{1+\delta}} \r)$ as $N\to\infty$. 
\end{remark}

\subsection{Connections to U-quantiles.}
\label{sec:u-stat}

In this section, we discuss connections of proposed algorithms to U-quantiles and the assumption requiring the groups $G_1,\ldots,G_k$ to be disjoint. 
We assume that the data $X_1,\ldots,X_N$ are i.i.d. with common distribution $P$, and let $\theta_\ast=\theta_\ast(P)\in \mb R$ be a real-valued parameter of interest. 
It is clear that the estimators produced by distributed algorithms considered above depend on the random partition of the sample. 
A natural way to avoid such dependence is to consider the U-quantile (in this case, the median) 
\begin{align*}
&
\widetilde \theta^{(k)} = \med{\bar\theta_J, \ J\in \m A_N^{(n)}},
\end{align*}
where $\m A_N^{(n)}:=\l\{  J: \ J\subseteq \{1,\ldots,N\}, \card(J)=n:=\lfloor N/k\rfloor \r\}$ is a collection of all distinct subsets of 
$\{1,\ldots,N\}$ of cardinality $n$, and $\bar\theta_J:=\bar\theta(X_j, \ j\in J)$ is an estimator of $\theta_\ast$ based on 
$\{X_j, \ j\in J\}$. 
For instance, when $\card(J)=2$ and $\bar\theta_J = \frac{1}{\card(J)}\sum_{j\in J} \frac{X_j}{2}$, $\widetilde \theta^{(k)}$ is the well-known Hodges-Lehmann estimator of the location parameter, see \citep{hodges1963estimates,lehmann2006nonparametrics}; for a comprehensive study of U-quantiles, see \citep{arcones1996bahadur}. 
The main result of this section is an analogue of Theorem \ref{th:main1} for the estimator $\widetilde \theta^{(k)}$; it implies that theoretical guarantees for the performance of $\widetilde \theta^{(k)}$ are at least as good as for the estimator 
$\widehat \theta^{(k)}$. 
Since the data are i.i.d., it is enough to impose the assumption \ref{ass:1} on $\bar\theta\l( X_1,\ldots,X_n \r)$ only, hence we drop the index $j$ and denote the normalizing sequence $\{\sigma_n\}_{n\in \mb N}$ and the corresponding error function $g(n)$. 
\begin{theorem}
\label{th:main1a}
Assume that $s>0$ and $n=\lfloor N/k \rfloor$ are such that 
\begin{align}
\label{eq:02}
&
g(n)+\sqrt{\frac{s}{k}}<\frac{1}{2}.
\end{align} 
Moreover, let assumption \ref{ass:1} be satisfied, and let $\zeta(n,s)$ solve the equation 
\[
\Phi \l( \zeta(n,s) \r) = \frac{1}{2} + g(n)+\sqrt{\frac{s}{k}}.
\] 
Then for any $s$ satisfying \eqref{eq:02},
\[
\l| \widetilde \theta^{(k)} - \theta_\ast \r| \leq \sigma_n \zeta(n,s)
\]
with probability at least $1 - 4e^{-2s}$.
\end{theorem}
\begin{proof}
See section \ref{sec:proof1a}. As before, a more explicit form of the bound immediately follows from Lemma \ref{cor:0}. 
\end{proof}
A drawback of the estimator $\widetilde \theta^{(k)}$ is the fact that its exact computation requires evaluation of $n\choose N$ estimators $\bar \theta_J$ over subsamples $\l\{ \{X_j, \ j\in J\}, \ J\in \m A_N^{(n)} \r\}$. 
For large $N$ and $n$, such task becomes intractable. 
However, an approximate result can be obtained by choosing $\ell$ subsets $J_1,\ldots,J_\ell$ from $\m A_N^{(n)}$ uniformly at random, and setting $\widetilde \theta_\ell^{(k)}:=\med{\bar\theta_{J_1},\ldots,\bar\theta_{J_\ell}}$. 
Typically, the error $\l| \widetilde \theta_\ell^{(k)} - \widetilde \theta^{(k)}\r|$ is of order $\ell^{-1/2}$ with high probability over the random draw of $J_1,\ldots,J_\ell$. 

We note that Theorem \ref{th:M-est} admits a similar extension for the estimator defined as 
\[
\widetilde \theta^{(k)}_\rho :=\argmin_{z\in \mb R} \sum_{J \in \m A_N^{(n)}} \rho\l( z - \bar \theta_J\r). 
\]
Namely, if the data are i.i.d., then under the assumptions of section \ref{ssec:m-estimators},
\begin{align}
\label{eq:03}
&
\l| \widetilde \theta^{(k)}_\rho - \theta_\ast \r| \leq 3 e^{\l( C_\rho/\sigma_n\r)^2} \cdot \sigma_n \l( \sqrt{\frac{s}{k}} + 2 g(n) \r)
\end{align}
with probability at least $1 - 4e^{-2s}$, whenever $s>0$ and $n=\lfloor N/k \rfloor$ are such that 
\[
e^{\l( C_\rho/\sigma_n\r)^2}\l( \sqrt{\frac{s}{k}} + 2 g(n) \r) \leq 0.33.
\]
We omit the proof of \eqref{eq:03} since the required modifications in the argument of Theorem \ref{th:M-est} are exactly the same as those explained in the proof of Theorem \ref{th:main1a}.

\section{Estimation in higher dimensions.}
\label{ssec:multivar}

In this section, it will be assumed that $\theta_\ast \in \mb R^m, \ m\geq 2$, is a vector-valued parameter of interest. 
Let $X_1,\ldots,X_N$ be independent $S$-valued random variables that are randomly partitioned into disjoint groups $G_1,\ldots,G_k$ of cardinality $n = \lfloor N/k \rfloor$ each. 
Let $\bar\theta_j:=\bar\theta_j(G_j)\in \mb R^m, \ 1\leq j\leq k$ be a sequence of estimators of $\theta_\ast$, the common parameter of the distributions of $X_j$'s.   
Assume that $\rho_1,\ldots,\rho_m$ are convex, even functions such that $\rho_i(z)\to\infty$ as $|z|\to\infty$ and $\|\rho_i' \|_\infty<\infty$, with $\rho_i'(x)$ defined as the average of the right and left derivatives of $\rho_i$, $i=1,\ldots,m$, and let 
\begin{align}
\label{eq:mvar-med}
\widehat\theta^{(k)} := \argmin_{z\in \mb R^m} \sum_{j=1}^k \sum_{i=1}^m \rho_i\l(  z_i - \bar\theta_{j,i}  \r), 
\end{align}
where $z=(z_1,\ldots,z_m)$ and $\bar\theta_j = (\bar\theta_{j,1},\ldots,\bar\theta_{j,m})$ for $1\leq j\leq k$.

For the sake of clarity, we will assume below that $X_1,\ldots,X_N$ are i.i.d. However, results can be easily extended to the case of non-identically distributed data in a manner described in section \ref{ssec:m-estimators}. 
Assumption \ref{ass:1} will be required to hold coordinatewise, namely, we will assume that there exist sequences 
$\{\sigma_{n,i}\}_{n\in \mb N}\subset \mb R_+, \ i=1,\ldots,m$, such that 
\[
g_m(n):=\max_{i=1,\ldots,m}\sup_{t\in \mb R}\l| \pr{ \frac{\bar\theta_{1,i} - \theta_\ast}{\sigma_{n,i}}\leq t} - \Phi( t ) \r| \to 0 \text{ as } n\to\infty.
\]
Note that the maximum over the second index $j$ disappears due to the i.i.d. assumption. 
\begin{theorem}
\label{th:mvar}
Let $C_{\rho_i}>0$ be such that $|\rho_{+,i}'(x)|\geq \frac{\|\rho_{+,i}'\|_\infty}{2}$ and $|\rho_{-,i}'(x)|\geq \frac{\|\rho_{-,i}'\|_\infty}{2}$ for $|x| > C_{\rho_i}, \ i=1,\ldots,m$. 
Let assumption \ref{ass:1} hold for each coordinate of $\bar\theta_{1}$, and suppose that $s>0$ and $n=\lfloor N/k \rfloor$ are such that 
\begin{align}
\label{eq:mvar-bound}
&
\max_{i=1,\ldots,m } e^{\l( C_{\rho_i}/\sigma_{n,i}\r)^2}\l( \sqrt{\frac{s}{k}} + 2 g_m(n) \r) \leq 0.33.
\end{align} 
Then for all $s$ satisfying \eqref{eq:mvar-bound} and all $1\leq i\leq m$ simultaneously,
\begin{align}
\label{eq:th5-bound}
&
\l| \widehat\theta_i^{(k)} - \theta_{\ast,i} \r| \leq 3 e^{\l( C_{\rho_i}/\sigma_{n,i}\r)^2} \cdot\sigma_{n,i}\l( \sqrt{\frac{s}{k}} + 2 g_m(n) \r)
\end{align}
with probability at least $1 - 4me^{-2s}$.
\end{theorem}
\begin{proof}
See section \ref{sec:proof-mvar}. 
\end{proof}

\subsection{Example: multivariate median-of-means estimator.}
\label{ssec:mmom}

Consider the special case of Theorem \ref{th:mvar} when $\theta_\ast=\mb E X$ is the mean of $X\in \mb R^m$, 
$\bar\theta_j(X):=\frac{1}{|G_j|}\sum_{X_i\in G_j} X_i$ is the sample mean evaluated over the subsample $G_j$, and $\rho_i(x) = |x|$ for all $i$. 
In this case, $\widehat\theta^{(k)}$ becomes the spatial median with respect to the $L_1$-norm, namely,
\begin{align}
\label{eq:med-L1}
&
\widehat\theta^{(k)} := \argmin_{z\in \mb R^m} \sum_{j=1}^k \l\|  z - \bar\theta_j \r\|_1.
\end{align}
The problem of finding the mean estimator that admits sub-Gaussian concentration around $\mb EX$ under weak moment assumptions on the underlying distribution has recently been investigated in several works. 
For instance, \citet{joly2016estimation} construct an estimator that admits ``almost optimal'' behavior under the assumption that the entries of $X$ possess 4 moments. 
Recently, \citet{lugosi2017sub,lugosi2018near} proposed new estimators that attains optimal bounds and requires existence of only 2 moments. 
More specifically, the aforementioned papers show that, for any $s$ such that $\frac{2}{N}<e^{-s}<1$, there exists an estimator $\hat\theta_{(s)}$ such that with probability at least $1 - C_1e^{-s}$, 
\[
\l\| \hat\theta_{(s)} - \theta_\ast \r\|_2\leq C_2\l( \sqrt{\frac{\tr( \Sigma)}{N}} + \sqrt{\frac{s\,\lambda_{\mathrm{max}}(\Sigma)}{N}}\r),
\]
where $C_1,C_2>0$ are numerical constants, $\Sigma$ is the covariance matrix of $X$, $\tr(\Sigma)$ is its trace and $\lambda_{\mathrm{max}}(\Sigma)$ - its largest eigenvalue. 
However, construction of these estimators explicitly depends on the desired  confidence level $s$, and (more importantly) they are numerically difficult to compute.  

On the other hand, Theorem \ref{th:mvar} demonstrates that performance of the multivariate median-of-means estimator is robust with respect to the choice of the number of subgroups $k$, and the resulting deviation bounds hold simultaneously over the range of confidence parameter $s$ whenever the coordinates of $X$ possess $2+\delta$ moments for some $\delta>0$. The following corollary summarizes these claims.
\begin{corollary}
\label{corollary:med-of-means2}
Let $X_1,\ldots,X_N$ be i.i.d. random vectors such that $\theta_\ast = \mb EX_1$ is the unknown mean, 
$\Sigma = \mb E\l[(X_1-\theta_\ast)(X_1-\theta_\ast)^T\r]$ is the covariance matrix, $\sigma_i^2 = \Sigma_{i,i}$, and $\max_{i=1,\ldots,m}\mb E | X_{1,i} |^{2+\delta}<\infty$ for some $\delta\in (0,1]$. 
Then there exist absolute constants $c_1,c_2>0$ such that for all $s>0$ and $k$ satisfying
\[
\sqrt{\frac{s}{k}} + \max_{i=1,\ldots,m}\frac{\mb E | X_{1,i} - \theta_{\ast,i}|^3}{\sigma_i^3 \sqrt{n}}\leq c_1,
\] 
with probability at least $1 - 4me^{-2s}$ for all $i=1,\ldots,m$ simultaneously,
\begin{align*}
&
\l| \widehat\theta_i^{(k)} - \theta_{\ast,i} \r| \leq c_2 \,\sigma_{i}\l( \max_{i=1,\ldots,m} \frac{\mb E | X_{1,i} - \theta_{\ast,i}|^{2+\delta}/\sigma_i^{2+\delta}}{n^{\frac{1+\delta}{2}}} + \sqrt{\frac{s}{N}}  \r).
\end{align*}
\end{corollary}
\begin{proof}
It follows from fact \ref{th:BE-gen} in section \ref{sec:prelim} that $g_m(n)$ can be bounded as 
\[
g_m(n) \leq A\max_{i=1,\ldots,m} \frac{\mb E|X_{1,i} - \theta_{\ast,i} |^{2+\delta}}{\sigma_i^{2+\delta}n^{\delta/2}}
\] 
for an absolute constant $A>0$. Moreover, it is easy to see that $C_{\rho_i} = 0$ for all $i$ and that assumption \ref{ass:1} holds with $\sigma_{n,i}=\frac{\sigma_i}{\sqrt{n}}$. Now the claim immediately follows from Theorem \ref{th:mvar}.
\end{proof}

\begin{remark}
Estimator \eqref{eq:med-L1} admits a natural generalization of the form
\begin{align}
\label{eq:geom-med}
\widehat\theta_{\rho,\|\cdot\|_\circ}^{(k)} := \argmin_{z\in \mb R^m} \sum_{j=1}^k \rho\l(\l\| z - \bar\theta_j \r\|_\circ \r), 
\end{align}
where $\|\cdot\|_\circ$ is a norm in $\mb R^m$ and $\rho$ is a convex, non-decreasing function. 
For example, if $\|\cdot\|_\circ$ is the Euclidean norm, resulting estimator is invariant with respect to the orthogonal transformations. However, available performance guarantees for this estimator hold under stronger assumptions (such as joint asymptotic normality of the coordinates of $\bar\theta_j$'s instead of coordinate-wise asymptotic normality), and exhibit suboptimal dependence on the dimension; these results, along with the discussion of relevant numerical methods, are presented in Appendix \ref{sec:l2-median}. 
Complete characterization of the effect of the norm $\|\cdot\|_\circ$ on the geometry of the problem and performance of the corresponding estimator \eqref{eq:geom-med} warrants further study. 
\end{remark}

\section{Simulation results.}
\label{ssec:simul}

We illustrate results of the previous sections with numerical simulations that compare performance of the median-of-means estimator with the usual sample mean, see figure \ref{fig:comparison} below. 
\begin{figure}[h]
    \centering
    \subfloat[]{
       \includegraphics[width=0.45\textwidth]{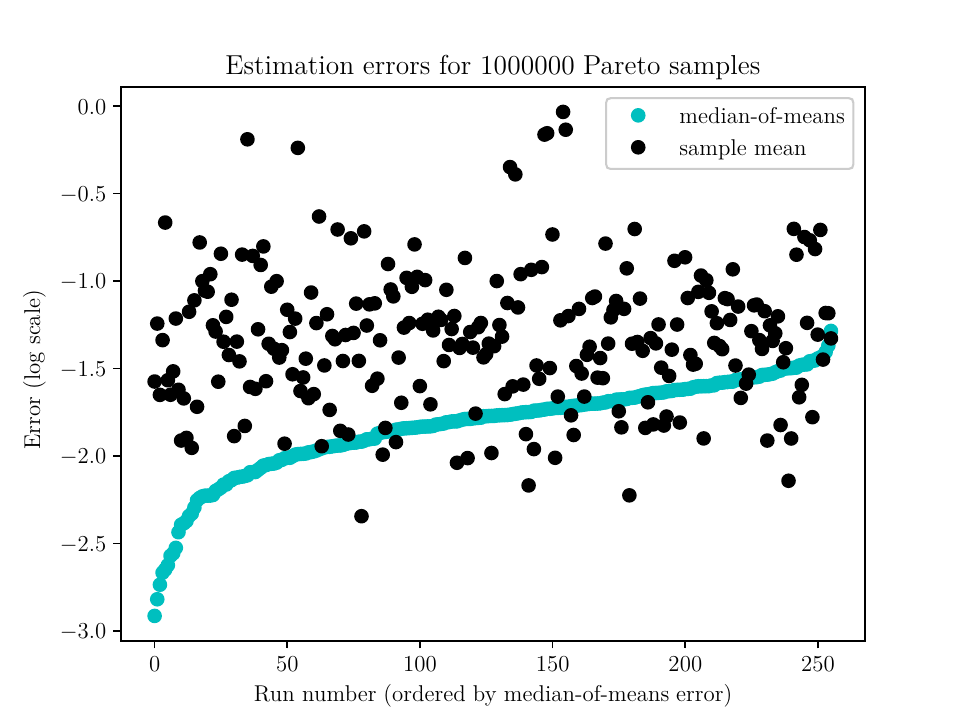}
    }%
    \subfloat[]{
        \includegraphics[width=0.45\textwidth]{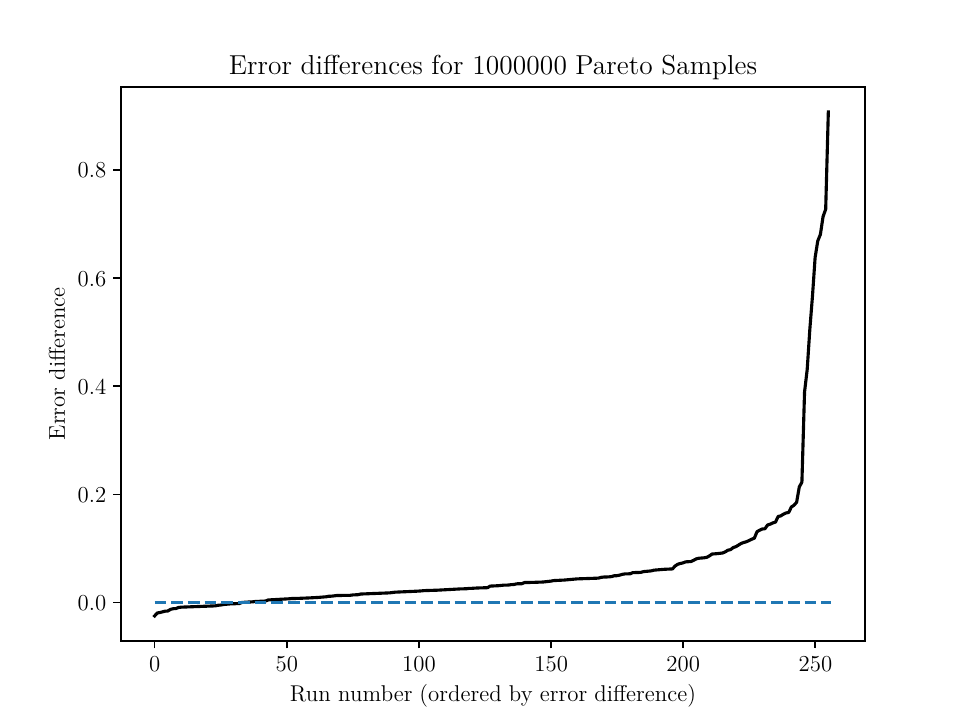}
        }  \\
          \subfloat[]{
        \includegraphics[width=0.45\textwidth]{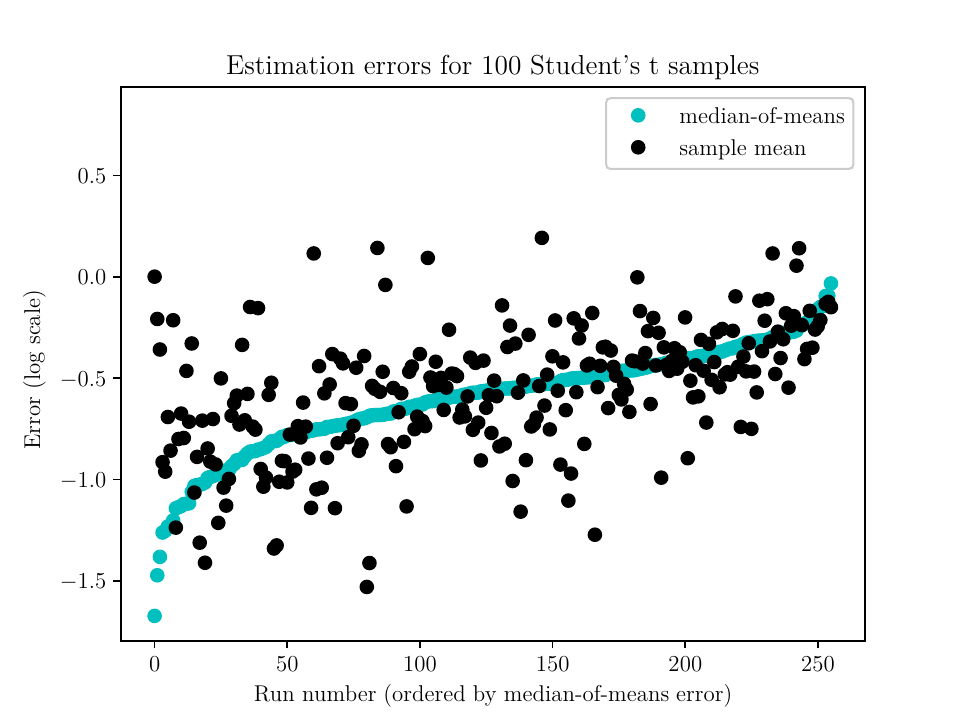}
    }
    \subfloat[]{
        \includegraphics[width=0.45\textwidth]{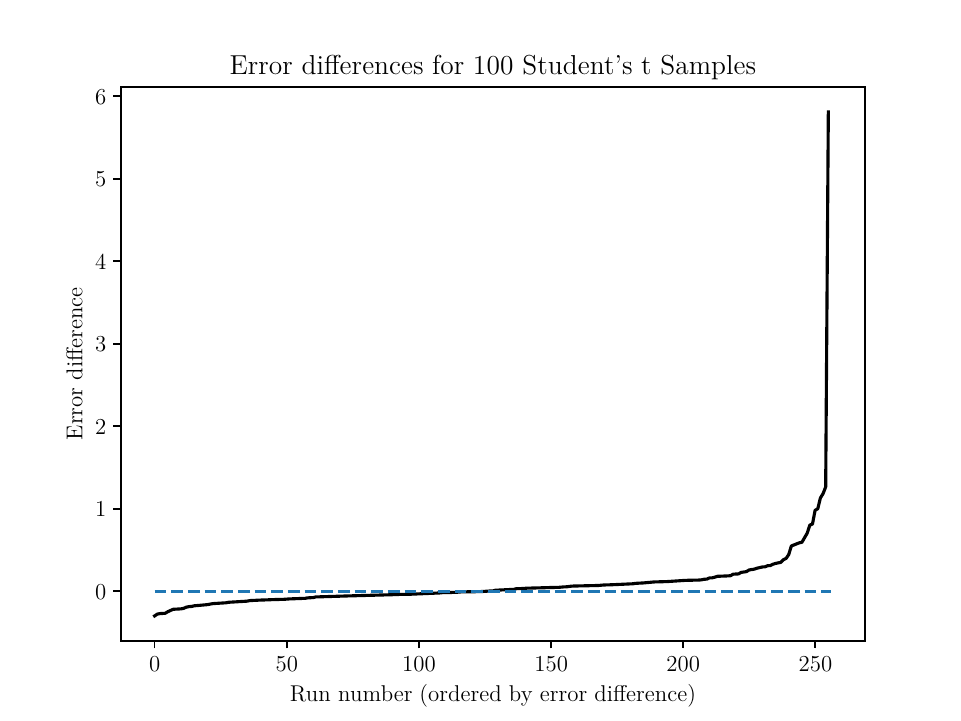}
        }
        \caption{Comparison of errors corresponding to the median-of-means and sample mean estimator over 256 runs of the experiment. In (a) the sample of size $N=10^6$ consists of i.i.d. random vectors in $\mb R^2$ with independent Pareto-distributed entries possessing only $2.1$ moments. Each run computes the median-of-means estimator using partition into $k=1000$ groups, as well as the usual sample mean. In (b), the ordered differences between the error of the sample mean and the median-of-means over all 256 runs illustrates robustness. Positive error differences in (b) indicate lower error for the median-of-means, and negative error differences occur when the sample mean provided a better estimate. \\
 Images (c) and (d) illustrate a similar experiment that was performed for two-dimensional random vectors with independent entries with Student's t-distribution with 2 degrees of freedom. In this case, the sample size is $N=100$ and the number of groups is $k=10$.
 }
    \label{fig:comparison}
\end{figure}
\begin{figure}[ht]
    \centering
         \includegraphics[width=0.55\textwidth]{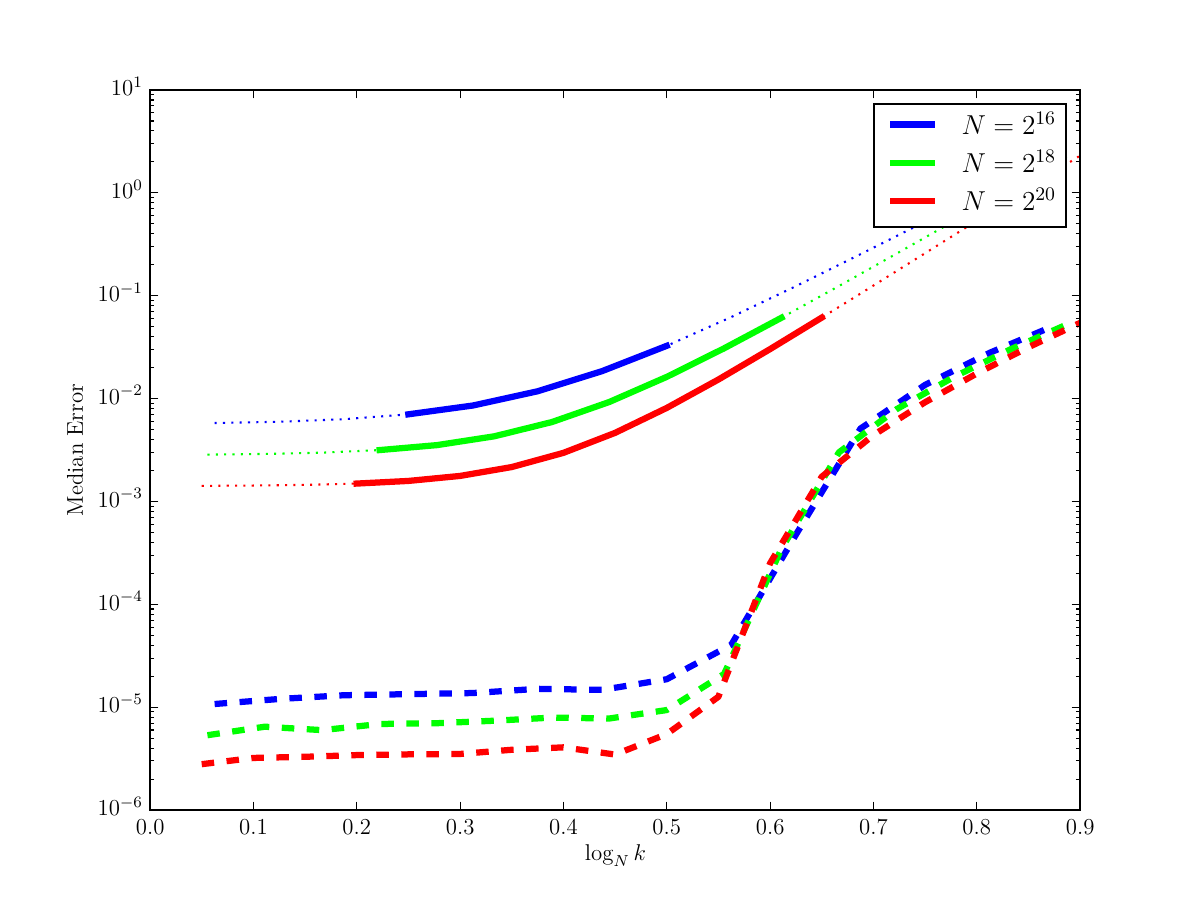}
        \caption{The solid and dotted lines indicate theoretical bounds for the different values of the sample size $N$, with the solid part indicating the number of subgroups $k$ for which our estimates hold. The dashed lines indicate empirical error between the median-of-means estimator and the true mean. We consider three cases: $N=2^{16}$ (blue), $N=2^{18}$ (green), and  $N=2^{20}$ (red). The $x$-axis is $\log_N k$ taken from a uniform partition of $(0,1)$ and the $y$-axis indicates the median error of the median-of-means estimator over $2^{16}$ runs of the experiment. For each value of $N$ and $k$, we run $2^{16}$ simulations by drawing $N$ i.i.d. random variables with Lomax distribution with shape parameter $\alpha=4$ and scale parameter $\lambda=1$, splitting into $k$ groups, and then computing the median of the means of those groups. From the $2^{16}$ simulations, we display (on a logarithmic scale) the median of the absolute differences between the true mean $1/3$ and the median-of-means estimators, producing the dashed lines in the figure. The solid and dotted lines are our theoretical bounds with $4e^{-2s}=1/2$ (that is, the probability that the solid and dotted bounds holds is guaranteed to be at least $1/2$).  }
    \label{fig:pareto}
\end{figure}
Moreover, we compared the theoretical guarantees for the median-of-means estimator (described in section \ref{ssec:med-of-means}) against the empirical outcomes for the Lomax distribution with shape parameter $\alpha=4$ and scale parameter $\lambda=1$; the corresponding probability density function is 
\[
p(x)=\frac{\alpha}{\lambda}\left(1+\frac{x}{\lambda}\right)^{-(\alpha+1)}\text{ for }x\geq 0
\]
In particular, the Lomax distribution with $\alpha=4$ and $\lambda=1$ has mean $1/3$ and median $\sqrt[4]{2}-1\approx 0.1892$. 
Since the mean and median do not coincide, the error of the median-of-means estimator has a significant bias component for large values of $k$. Figure \ref{fig:pareto} depicts the impact of the bias beyond $k=\sqrt{N}$ (equivalently, $\log_N k=1/2$), and also the fact that the median error is mostly flat for $k < \sqrt{N}$. 

Finally, we assessed empirical coverage of the confidence intervals constructed using Theorem \ref{th:clt} and centered at the median-of-means estimator; results are presented in figure \ref{fig:confidence}. The sample of size $N=10^5$ was generated from the half-t distribution with $3$ degrees of freedom; recall that a random variable $\xi$ has half-t distribution with $\nu$ degrees of freedom if $\xi \stackrel{\mathrm{d}}{=}|\eta|$ where $\eta$ has usual t-distribution with $\nu$ degrees of freedom. It is clear that half-t distribution is both asymmetric and heavy-tailed. Each sample was further corrupted by outliers sampled from the normal distribution with mean $0$ and standard deviation $10^5$; the number of outliers ranged from $0$ to $\sqrt{N}=100$ with increments of $20$. The median-of-means estimator was constructed for $k=\sqrt{N}=100$. For comparison, we present empirical coverage levels attained by the sample mean in the same framework. 
\begin{figure}
\centering
 \subfloat[]{
\begin{tabular}[ht]{c|c|c|c|c|c|c}
\hline
\text{Nominal confidence level} & \multicolumn{6}{c}{Fraction of outliers} \\
\hline
\mbox{ } & 0 & $\frac{0.2}{\sqrt{N}}$ & $\frac{0.4}{\sqrt{N}}$ & $\frac{0.6}{\sqrt{N}}$ & $\frac{0.8}{\sqrt{N}}$ & $\frac{1}{\sqrt{N}}$ \\ 
\hline
0.8   & 0.94 & 0.0008 & 0& 0& 0 & 0 \\
0.95 & 0.99 & 0.001 & 0 & 0& 0 & 0 \\
\hline
\end{tabular}
}
\\
 \subfloat[]{
\begin{tabular}[ht]{c|c|c|c|c|c|c}
\hline
\text{Nominal confidence level} & \multicolumn{6}{c}{Fraction of outliers} \\
\hline
\mbox{ } & 0 & $\frac{0.2}{\sqrt{N}}$ & $\frac{0.4}{\sqrt{N}}$ & $\frac{0.6}{\sqrt{N}}$ & $\frac{0.8}{\sqrt{N}}$ & $\frac{1}{\sqrt{N}}$ \\ 
\hline
0.8   & 0.88 & 0.82 & 0.77 & 0.66 & 0.6 & 0.53 \\
0.95 & 0.99 & 0.97 & 0.93 & 0.85 & 0.79 & 0.71\\
\hline
\end{tabular}
}
\caption{Empirical coverage levels of confidence intervals constructed using (a) the Central Limit Theorem for the sample mean and (b) Theorem \ref{th:clt} for the median of means; (a) reflects the results obtained for the sample mean and (b) reflects the results obtained for the median of means estimator.}
\label{fig:confidence}
\end{figure}

\section{Proofs}
\label{sec:proofs}

In this section, we present the proofs of the main results.

\subsection{Preliminaries.}
\label{sec:prelim}

We recall several facts that are used in the proofs below. 
The following bound has been established by A. Berry \citep{berry1941accuracy} and C.-G. Esseen \citep{esseen1942liapounoff}. 
A version with an explicit constant given below is due to \citet{shevtsova2011absolute}.
\begin{fact}[Berry-Esseen bound]
\label{th:BE}
Assume that $Y_1,\ldots,Y_n$ is a sequence of i.i.d. copies of a random variable $Y$ with mean $\mu$, variance $\sigma^2$ and such that $\mb E|Y|^3<\infty$. Then
\[
\sup_{s\in \mb R}\l| \pr{\sqrt{n}\frac{\bar Y - \mu}{\sigma} \leq s} - \Phi(s) \r| \leq 0.4748 \frac{\mb E|Y -\mu |^3}{\sigma^3\sqrt{n}},
\]
where $\bar Y = \frac{1}{n}\sum_{j=1}^n Y_j$ and $\Phi(s)$ is the cumulative distribution function of the standard normal random variable. 
\end{fact}
The following generalization of Berry-Esseen bound is due to \cite{petrov1995limit}. 
\begin{fact}[Generalization of Berry-Esseen bound]
\label{th:BE-gen}
Assume that $Y_1,\ldots,Y_n$ is a sequence of i.i.d. copies of a random variable $Y$ with mean $\mu$, variance $\sigma^2$ and such that $\mb E|Y|^{2+\delta}<\infty$ for some $\delta\in(0,1]$. Then there exists an absolute constant $A>0$ such that
\[
\sup_{s\in \mb R}\l| \pr{\sqrt{n}\frac{\bar Y - \mu}{\sigma} \leq s} - \Phi(s) \r| \leq A \frac{\mb E|Y -\mu |^{2+\delta}}{\sigma^{2+\delta}n^{\delta/2}}.
\]
\end{fact}

Next, we recall a well-known concentration inequality. 		
\begin{fact}[Bounded difference inequality]
\label{th:BDI}
Let $X_1,\ldots,X_k$ be i.i.d. random variables, and assume that $Z=g(X_1,\ldots,X_k)$, where $g$ is such that 
for all $j=1,\ldots,k$ and all $x_1,x_2,\ldots,x_j,x_j',\ldots,x_k$,
\[
\l| g(x_1,\ldots,x_{j-1},x_j,x_{j+1},\ldots,x_k) - g(x_1,\ldots,x_{j-1},x_j',x_{j+1},\ldots,x_k)\r| \leq c_j.
\]
Then 
\[
\pr{ Z-\mb EZ \geq t}\leq \exp\l\{ -\frac{2t^2}{\sum_{j=1}^k c_j^2} \r\}
\]
and 
\[
\pr{ Z-\mb EZ \leq -t}\leq \exp\l\{ -\frac{2t^2}{\sum_{j=1}^k c_j^2} \r\}.
\]
\end{fact}
Finally, we recall the definition of a U-statistic. 
Let $h: \mb R^n\mapsto \mb R$ be a measurable function of $n$ variables, and 
\[
\m A_N^{(n)}:=\l\{  J: \ J\subseteq \{1,\ldots,N\}, \card(J)=n \r\}.
\] 
A U-statistic of order $n$ with kernel $h$ based on the i.i.d. sample $X_1,\ldots,X_N$ is defined as \citep{hoeffding1948class}
\[
U_N(h)=\frac{1}{{n\choose N}}\sum_{J\in \m A_N^{(n)}} h \l( X_j, \ j \in J \r).
\]
Clearly, $\mb E U_N (h)=\mb Eh(X_1,\ldots,X_n)$, moreover, $U_N (h)$ has the smallest variance among all unbiased estimators. The following analogue of fact \ref{th:BDI} holds for the U-statistics:
\begin{fact}[Concentration inequality for U-statistics, \citep{hoeffding1963probability}]
\label{th:U-S}
\mbox{ }\\
Assume that the kernel $h$ satisfies $\l| h(x_1,\ldots,x_n) \r| \leq M$ for all $x_1,\ldots,x_n$. Then for all $s>0$,
\[
\pr{\l| U_N(h) - \mb E U_N(h) \r| \geq s}\leq 2 \exp\l\{ -\frac{2 \lfloor N/n \rfloor t^2}{M^2} \r\}.
\]
\end{fact}
		
\subsection{Proof of Theorem \ref{th:main1}.}
\label{sec:proof1}

Observe that 
\[
\l| \widehat \theta^{(k)} - \theta_\ast \r| = 
\l| \med{\bar \theta_1 - \theta_\ast, \ldots, \bar \theta_k - \theta_\ast} \r|.
\]
Let $\Phi^{(n_j,j)}(\cdot)$ be the distribution function of $\bar \theta_j - \theta_\ast, \ j=1,\ldots,k,$ and $\widehat\Phi_k(\cdot)$ - the empirical distribution function corresponding to the sample $W_1=\bar \theta_1 - \theta_\ast, \ldots, W_k=\bar \theta_k - \theta_\ast$, that is, 
\[
\widehat\Phi_k(z) = \frac{1}{k}\sum_{j=1}^k I \l\{ W_j \leq z\r\}.
\] 
Suppose that $z\in\R$ is fixed, and note that $\widehat\Phi_k(z)$ is a function of the random variables $W_1,\ldots, W_k$, and 
$\mb E\widehat\Phi_k(z) = \frac{1}{k}\sum_{j=1}^k \Phi^{(n_j,j)}(z)$. 
Moreover, the hypothesis of the bounded difference inequality (fact \ref{th:BDI}) is satisfied with $c_j=1/k$ for $j=1,\ldots, k$, and therefore it implies that
\begin{align}
\label{eq:b00}
\l| \widehat\Phi_k(z) - \frac{1}{k}\sum_{j=1}^k \Phi^{(n_j,j)}(z) \r| \leq \sqrt{\frac{s}{k}}
\end{align}
on the draw of $W_1,\ldots, W_k$ with probability $\geq 1 - 2e^{-2s}$. 

Let $z_1\geq z_2$ be such that 
$\frac{1}{k}\sum_{j=1}^k \Phi^{(n_j,j)}(z_1) \geq \frac{1}{2} +\sqrt{\frac{s}{k}}$ and $\frac{1}{k}\sum_{j=1}^k \Phi^{(n_j,j)}(z_2) \leq \frac 1 2 - \sqrt{\frac{s}{k}}$. 
Applying \eqref{eq:b00} for $z=z_1$ and $z=z_2$ together with the union bound, we see that for $j=1,2$,
\[
\l| \widehat\Phi_k(z_j) - \frac{1}{k}\sum_{j=1}^k \Phi^{(n_j,j)}(z_j) \r| \leq \sqrt{\frac{s}{k}}
\]
on an event $\m E$ of probability $\geq 1 - 4e^{-2s}$. 
It follows that on $\m E$, $\widehat\Phi_k(z_1)\geq 1/2$ and $1 - \widehat\Phi_k(z_2)\geq 1/2$ simultaneously, hence 
\begin{align}
\label{eq:10}
&
\med{W_1,\ldots,W_k}\in [z_2,z_1]
\end{align}
by the definition of the median. 
It remains to estimate $z_1$ and $z_2$. 
Assumption \ref{ass:1} implies that
\begin{multline*}
\frac{1}{k}\sum_{j=1}^k \Phi^{(n_j,j)}(z_1) \geq \frac{1}{k}\sum_{j=1}^k \Phi\l( \frac{z_1}{\sigma_{n_j}^{(j)}} \r) - 
\l| \frac{1}{k}\sum_{j=1}^k \l( \Phi^{(n_j,j)}(z_1) - \Phi \l( \frac{z_1}{\sigma_{n_j}^{(j)}} \r) \r) \r| 
\\
\geq \frac{1}{k}\sum_{j=1}^k \Phi\l( \frac{z_1}{\sigma_{n_j}^{(j)}} \r) - \frac{1}{k}\sum_{j=1}^k g_j(n_j). 
\end{multline*}
Hence, it suffices to find $z_1$ such that 
$\frac{1}{k}\sum_{j=1}^k \Phi\l( \frac{z_1}{\sigma_{n_j}^{(j)}} \r) \geq \frac{1}{2} +\sqrt{\frac{s}{k}}+  \frac{1}{k}\sum_{j=1}^k g_j(n_j)$. 
Recall that 
$\alpha_j=\frac{ 1/\sigma_{n_j}^{(j)} }{1/k\sum_{i=1}^k 1/\sigma_{n_j}^{(i)}}, \ j=1,\ldots,k$, 
and let $\zeta_{j}(n_j,s)$ be the solution of the equation 
\[
\Phi\l( \zeta_{j}(n_j,s)/\sigma_n^{(j)} \r) - \frac{1}{2} = \alpha_j\cdot \frac{1}{k}\sum_{i=1}^k \l( g_i(n_i) + \sqrt{\frac{s}{k}} \r).
\] 
Note that $\zeta_{j}(n,s)$ always exists since $\alpha_j \cdot\frac{1}{k}\sum_{i=1}^k \l( g_i(n_i) + \sqrt{\frac{s}{k}} \r)< \frac{1}{2}$ by assumption. 
Finally, since $\sum_{j=1}^k \alpha_j = k$, it is clear that any 
\[
z_1 \geq \max_{j=1,\ldots,k} \zeta_{j}(n_j,s)
\]
satisfies the requirements. 
Similarly, 
\begin{multline*}
\frac{1}{k}\sum_{j=1}^k\Phi^{(n_j,j)}(z_2)\leq  \frac{1}{k}\sum_{j=1}^k \Phi\l( \frac{z_2}{\sigma_{n_j}^{(j)}} \r) +
\l| \frac{1}{k}\sum_{j=1}^k \l( \Phi^{(n_j,j)}(z_2) - \Phi \l( \frac{z_2}{\sigma_{n_j}^{(j)}} \r) \r) \r| 
\\
\leq \frac{1}{k}\sum_{j=1}^k \Phi\l( \frac{z_2}{\sigma_{n_j}^{(j)}} \r) + \frac{1}{k}\sum_{j=1}^k g_j(n_j)
\end{multline*}
by assumption \ref{ass:1}, hence it is sufficient to choose $z_2$ such that 
$z_2\leq \max_{j=1,\ldots,k}\tilde\zeta_{j}(n_j,s)$, where $\tilde\zeta_{j}(n_j,s)$ satisfies 
$\Phi\l( \tilde\zeta_{j}(n_j,s)/\sigma_n^{(j)} \r) - \frac{1}{2} =  - \alpha_j\cdot \frac{1}{k}\sum_{i=1}^k \l( g_i(n_i) + \sqrt{\frac{s}{k}} \r)$.
Noting that $\tilde\zeta_{j}(n_j,s)=-\zeta_{j}(n_j,s)$ and recalling \eqref{eq:10}, we conclude that 
\[
\l| \widehat \theta^{(k)} - \theta_\ast \r| \leq \max_{j=1,\ldots,k} \zeta_{j}(n_j,s)
\]
with probability at least $1 - 4e^{-2s}$. 

\subsection{Proof of Theorem \ref{th:M-est}.}
\label{sec:proof-M}

We will use notation as in the proof of Theorem \ref{th:main1}. 
Clearly, $\widehat\theta_\rho^{(k)}$ satisfies the equation $G(\widehat\theta_\rho^{(k)})=0$, where 
\[
G(z)=\frac{1}{k}\sum_{j=1}^k \rho'\l(z - \bar\theta_j\r).
\]
Suppose $z_1,z_2$ are such that $G(z_1) > 0$ and $G(z_2) < 0$. 
Since $G$ is increasing, it is easy to see that $\widehat\theta_\rho^{(k)}\in (z_2,z_1)$. 
To find such $z_1$ and $z_2$, we proceed in 3 steps. 

(a) First, observe that the bounded difference inequality (fact \ref{th:BDI}) implies that for any fixed $z\in \mb R$,
\begin{align*}
&
\frac{1}{k}\l| \sum_{j=1}^k 
\Big( \rho'\l( z - \bar\theta_j\r) - \mb E \rho'\l( z - \bar\theta_j\r) \Big) \r| \leq 
\l\|\rho' \r\|_\infty \sqrt{\frac{s}{k}}
\end{align*}
with probability $\geq 1 - 2e^{-2s}$. 

(b) Next, we will find an upper bound for 
\[
\frac{1}{k}\l| \sum_{j=1}^k 
\l( \mb E \rho'\l( z - \bar\theta_j\r)  -
\mb E \rho'\l( z - Z_j\r) \r)\r|,
\]
where $Z_j\sim N\l( \theta_\ast,\l( \sigma_{n_j}^{(j)}\r)^2 \r), \ j=1,\ldots,k$ are independent. 
Note that for any bounded non-negative function $f:\mb R\mapsto \mb R_+$ and a signed measure $Q$, 
\begin{align*}
\l| \int_{\mb R} f(x)dQ \r| = \l| \int_{0}^{\|f\|_\infty} Q\l( x: \,f(x)\geq t \r) dt \r|
\leq \|f\|_\infty \max_{t\geq 0} \l|  Q\l( x: \,f(x)\geq t \r)\r|.
\end{align*} 
Since any bounded function $f$ can be written as $f = \max(f,0) - \max(-f, 0)$, we deduce that 
\[
\l| \int_{\mb R} f(x)dQ \r|\leq \|f\|_\infty \l( \max_{t\geq 0} \l|  Q\l( x: \,f(x)\geq t \r)\r| + \max_{t\leq 0} \l|  Q\l( x: \,f(x)\leq t \r)\r|\r).
\] 
Moreover, if $f$ is monotone, the sets $\{x: \,f(x)\geq t\}$ and $\{x: \,f(x)\leq t\}$ are half-intervals. 
Applying this to $f=\rho'$ and $Q=\Phi^{(n_j,j)} - \Phi$, we deduce that
\begin{align*}
\frac{1}{k}\l| \sum_{j=1}^k 
\l( \mb E \rho'\l( z - \bar\theta_j\r)  - 
\mb E \rho'\l( z - Z_j \r) \r)\r|
&\leq 
2\|\rho'\|_\infty \frac{1}{k}\sum_{j=1}^k \sup_{t\in \mb R} \l| \Phi^{(n_j,j)}(t) - \Phi(t) \r| 
\\
& \leq 2\|\rho'\|_\infty \, \frac{1}{k}\sum_{j=1}^k g_j(n_j)
\end{align*}
by assumption \ref{ass:1}.

(c) In remains to find $z_1$ satisfying
\[
\frac{1}{k}\sum_{j=1}^k \mb E\, \rho'\l( z_1 -\theta_\ast - (Z_j-\theta_\ast)\r)
> \l\|\rho' \r\|_\infty \l( \sqrt{\frac{s}{k}} + \frac{2}{k}\sum_{i=1}^k g_i(n_i) \r).
\]
Let $\tilde z_1:= z_1-\theta_\ast$ and $\tilde Z_j:=Z_j - \theta_\ast$. 
Since $\sum_{j=1}^k \alpha_j=k$ (where $\alpha_j$'s were defined in \eqref{eq:alpha}), it suffices to find $z_1$ such that 
$\mb E \rho'\l( \tilde z_1 - \tilde  Z_j\r) > \alpha_j\l\|\rho' \r\|_\infty \l( \sqrt{\frac{s}{k}} + \frac{2}{k}\sum_{i=1}^k g_i(n_i) \r)$ for all $j$. 
For any bounded function $h$ such that $h(-x)=-h(x)$ and $h(x)\geq 0$ for $x\geq 0$, and any $z\geq 0$,
\[
\int_\mb R h(x+z)\phi(x) dx = \int_0^\infty h(x)\l( \phi(x-z) - \phi(-x-z)\r)dx \geq 0, 
\] 
where $\phi(x) = (2\pi)^{-1/2}e^{-x^2/2}$. 
Recall that $C_\rho>0$ is such that $|\rho'(x)|\geq \|\rho'\|_\infty/2$ for $|x|\geq C_\rho$. It follows that
\begin{align*}
\mb E \rho'\l( \tilde  z_1 - \tilde Z_j\r) &\geq 
\frac{1}{2}\|\rho'\|_\infty \mb E \Big( I\{ \tilde z_1 - \tilde Z_j \geq C_\rho \} - I\{ \tilde z_1 - \tilde Z_j \leq - C_\rho\} \Big) 
\\
& = \frac{1}{2}\|\rho'\|_\infty \Big( \pr{\tilde Z_j \geq C_\rho - \tilde z_1} - \pr{\tilde Z_j \leq -C_\rho - \tilde z_1}\Big)
\\
& = \frac{1}{2}\|\rho'\|_\infty \pr{Z \in \l[ \frac{C_\rho - \tilde z_1}{\sigma_{n_j}^{(j)}}, \frac{C_\rho + \tilde z_1}{\sigma_{n_j}^{(j)}}\r]},
\end{align*}
where $Z\sim N(0,1)$. Next, Lemma \ref{lemma:normal} implies that 
\[
\pr{Z \in \l[ \frac{C_\rho - \tilde z_1}{\sigma_{n_j}^{(j)}}, \frac{C_\rho + \tilde z_1}{\sigma_{n_j}^{(j)}}\r]}
\geq 2 e^{-\l( C_\rho/\sigma_{n_j}^{(j)}\r)^2}\pr{Z\in \l[ 0,\tilde z_1/\sigma_{n_j}^{(j)} \r] }.
\]
Combining the previous two bounds, we deduce that it suffices to find $\tilde z_1>0$ such that 
\[
\pr{Z\in [0,\tilde z_1/\sigma_{n_j}^{(j)} } \geq \alpha_j \,e^{\l( C_\rho/\sigma_{n_j}^{(j)}\r)^2}\l( \sqrt{\frac{s}{k}} + \frac{2}{k}\sum_{i=1}^k g_i(n_i) \r).
\]
By our assumptions, $\max_{j=1,\ldots,k} \alpha_j \,e^{\l( C_\rho/\sigma_{n_j}^{(j)}\r)^2}\l( \sqrt{\frac{s}{k}} + \frac{2}{k}\sum_{i=1}^k g_i(n_i) \r) \leq 0.33$. Lemma \ref{cor:0} yields that it suffices to take
\[
\tilde z_1 = z_1 - \theta_\ast = 3H_k \max_{j=1,\ldots,k} e^{\l( C_\rho/\sigma_{n_j}^{(j)}\r)^2}  \l( \sqrt{\frac{s}{k}} + \frac{2}{k}\sum_{i=1}^k g_i(n_i) \r).
\]
The estimate for $z_2$ follows the same pattern, and yields that one can take $z_2$ as 
\[
z_2 = \theta_\ast - 3H_k \max_{j=1,\ldots,k} e^{\l( C_\rho/\sigma_{n_j}^{(j)}\r)^2}  \l( \sqrt{\frac{s}{k}} + \frac{2}{k}\sum_{i=1}^k g_i(n_i) \r),
\]
implying the claim.
		
\subsection{Proof of Theorem \ref{th:clt}.}
\label{sec:proofclt}

Recall that $L(z)=\mb E\rho'(z+Z)$ for $Z\sim N(0,1)$, and note that under our assumptions, equation $L(z) = 0$ has a unique solution $z=0$ (even if $\rho$ is not strictly convex). Next, observe that 
\begin{align*}
\pr{ \sum_{j=1}^k \rho'_- \l(\frac{\theta_\ast - \bar\theta_j + \frac{t\Delta\sigma_n}{\sqrt{k}} }{\sigma_n}\r) <0 } & \leq
\pr{ \frac{\sqrt{k}}{\sigma_n}\l( \widehat\theta^{(k)}_\rho - \theta_\ast\r) \geq t\Delta  } \\
& \leq 
\pr{ \sum_{j=1}^k \rho'_- \l(\frac{\theta_\ast - \bar\theta_j + \frac{t\Delta\sigma_n}{\sqrt{k}} }{\sigma_n}\r) \leq 0 },
\end{align*}
hence it suffices to show that both the left-hand side and the right-hand side of the inequality above converge to 
$1-\Phi(t)$ for all $t$. We will outline the argument for the left-hand side, and the remaining part is proven in a similar fashion. Note that 
\begin{align}
\label{eq:clt}
\pr{ \sum_{j=1}^k \rho'_-\l(\frac{\theta_\ast - \bar\theta_j + \frac{t\Delta\sigma_n}{\sqrt{k}} }{\sigma_n}\r) <0 } =
\pr{ \frac{\sum_{j=1}^k Y_{n,j} - \mb EY_{n,j} }{\sqrt{k\var\l( Y_{n,1} \r)}} < -\frac{\sqrt{k}  \, \mb E Y_{n,1} }{\sqrt{\var\l(Y_{n,1} \r)} }}, 
\end{align}
where $Y_{n,j} = \rho'_-\l(\frac{\theta_\ast - \bar\theta_j + \frac{t\Delta\sigma_n}{\sqrt{k}} }{\sigma_n}\r)$. 

\begin{lemma}
\label{lemma:lim}
Under the assumptions of Theorem \ref{th:clt}, 
$\sqrt{k}\mb EY_{n,1}\to t\, \Delta\,L'(0)$ and 
\[
\sqrt{\var\l(Y_{n,1} \r)}~\to~\sqrt{\mb E\l( \rho'(Z) \r)^2} = \Delta \cdot L'(0) \text{ as } N\to\infty, 
\]
where $Z\sim N(0,1)$. 
\end{lemma}
\begin{proof}[of Lemma \ref{lemma:lim}]
Let $Z\sim N(0,1)$. Since $\rho$ is convex, its derivative $\rho':=(\rho'_+ + \rho'_-)/2$ is monotone and continuous almost everywhere (with respect to Lebesgue measure). 
Together with the assumption that $\|\rho' \|_\infty<\infty$, Lebesgue dominated convergence Theorem implies that 
\begin{multline}
\label{eq:differ}
\frac{d}{dz} L(z)\big|_{z=0} = \lim_{h\to 0} \frac{1}{h\sqrt{2\pi}}\int_\mb R \rho'(x+h) e^{-x^2/2} dx = \lim_{h\to 0} \frac{1}{h\sqrt{2\pi}}\int_\mb R \rho'(x) e^{-(x-h)^2/2}dx 
\\ 
= \frac{1}{\sqrt{2\pi}} \int_\mb R x\rho'(x) e^{-x^2/2}dx.
\end{multline}
Next, we will prove the assertion that $\sqrt{k}\mb EY_{n,1}\to t\, \Delta\,L'(0)$. 
It is easy to see that 
\begin{multline*}
\sqrt{k}\mb EY_{n,1} 
=\sqrt{k}
\l( \mb E \rho'\l( \frac{\theta_\ast - \bar\theta_1}{\sigma_n} + \frac{t\Delta}{\sqrt{k}} \r) - \mb E\rho'\l(Z + \frac{t\Delta}{\sqrt{k}} \r) \r) \\
 +  t\Delta\cdot \frac{1}{t\Delta/\sqrt{k}} \l( \mb E\rho'\l(Z + \frac{t\Delta}{\sqrt{k}} \r) -  \underbrace{\mb E\rho'\l(Z\r)}_{=0} \r).
\end{multline*}
Reasoning as in the proof of Theorem \ref{th:M-est}  (see step (b) in section \ref{sec:proof-M}), we deduce that 
\[
\l| \mb E \rho'\l( \frac{\theta_\ast - \bar\theta_1}{\sigma_n} + \frac{t\Delta}{\sqrt{k}} \r) - \mb E\rho'\l(Z + \frac{t\Delta}{\sqrt{k}} \r) \r| \leq 2\l\| \rho' \r\|_\infty\, g(n),
\]
where $g(n)$ is the function from assumption \ref{ass:1}. Hence, recalling that $g(n)\sqrt{k}\to 0$ as $N\to\infty$, we obtain that 
\[
\sqrt{k}
\l( \mb E \rho'\l( \frac{\theta_\ast - \bar\theta_1}{\sigma_n} + \frac{t\Delta}{\sqrt{k}} \r) - \mb E\rho'\l(Z + \frac{t\Delta}{\sqrt{k}} \r) \r) \to 0 \text{ as } N\to \infty.
\]
On the other hand, it follows from \eqref{eq:differ} that for $t\ne 0$
\[
t\Delta\cdot \frac{1}{t\Delta/\sqrt{k}} \mb E\rho'\l(Z + \frac{t\Delta}{\sqrt{k}} \r) 
\xrightarrow{N\to\infty} t\Delta\cdot L'(0).
\]
For $t=0$, it is also clear that $\mb E\rho'\l( Z \r) = 0$. 
To establish the fact that $\sqrt{\var\l(Y_{n,1} \r)}~\to~\sqrt{\mb E\l( \rho'(Z) \r)^2}$, note that weak convergence of 
$\frac{\bar\theta_1 - \theta_\ast}{\sigma_n}$ to the normal law (assumption \ref{ass:1}) together with Lebesgue dominated convergence Theorem implies that 
\begin{align*}
&\mb E \rho'\l(\frac{\theta_\ast - \bar\theta_1 }{\sigma_n} + \frac{t\Delta}{\sqrt{k}}\r) \to \mb E \rho'\l( Z\r)\ = 0,
\\
& \mb E\l( \rho'\l(\frac{\theta_\ast - \bar\theta_1 }{\sigma_n} + \frac{t\Delta}{\sqrt{k}}\r) \r)^2 \to \mb E \l( \rho'(Z) \r)^2.
\end{align*}
Since $L'(0)>0$, we deduce that 
\[
\mb E^{1/2} \l( \rho'(Z) \r)^2 = \Delta\cdot L'(0),
\]
and the claim follows.
\end{proof}

\noindent
Lemma \ref{lemma:lim} implies that 
$-\frac{\sqrt{k}  \, \mb E Y_{n,1} }{\sqrt{\var\l(Y_{n,1} \r)} }\xrightarrow{N\to\infty} t$. 
It remains to apply Lindeberg's Central Limit Theorem \citep[][Theorem 1.9.3]{serfling1981approximation} to $Y_{n,j}$'s to deduce the result from equation \eqref{eq:clt}. 
To this end, we only need to verify the Lindeberg condition requiring that for any $\eps>0$,
\begin{align}
\label{eq:g00}
&
\mb E (Y_{n,1} - \mb E Y_{n,1})^2 \,I\l\{ |Y_{n,1} - \mb EY_{n,1}|\geq \eps\sqrt{k}\r\} \to 0 \text{ as } k\to\infty.
\end{align}
However, since $\rho'(\cdot)$ (and hence $Y_{n,1}$) is bounded, \eqref{eq:g00} easily follows.

\subsection{Proof of Theorem \ref{th:main1a}.}
\label{sec:proof1a}

The argument is similar to the proof of Theorem \ref{th:main1}. 
Let $\Phi^{(n)}(\cdot)$ be the distribution function of $\frac{\bar \theta_1 - \theta_\ast}{\sigma_n}$ and 
$\widehat\Phi_{N\choose n}(\cdot)$ - the empirical distribution function corresponding to the sample 
$\l\{ W_J=\frac{\bar \theta_J - \theta_\ast}{\sigma_n}, \ J\in \m A_N^{(n)} \r\}$ of size $N \choose n$. 

Suppose that $z\in\R$ is fixed, and note that $\widehat\Phi_{N\choose n}(z)$ is a U-statistic with mean $\Phi^{(n)}(z)$. 
We will apply the concentration inequality for U-statistics (fact \ref{th:U-S}) with $M=1$ to get that 
\begin{align}
\label{eq:b01}
\l| \widehat\Phi_{N\choose n}(z) - \Phi^{(n)}(z) \r| \leq \sqrt{\frac{s}{\lfloor N/n \rfloor}} \leq \sqrt{\frac{s}{k}}
\end{align}
with probability $\geq 1 - 2e^{-2s}$; here, we also used the fact that $n=\lfloor N/k\rfloor$. 

Let $z_1\geq z_2$ be such that 
$\Phi^{(n)}(z_1) \geq \frac{1}{2} +\sqrt{\frac{s}{k}}$ and $\Phi^{(n)}(z_2) \leq \frac 1 2 - \sqrt{\frac{s}{k}}$. 
Applying \eqref{eq:b01} for $z=z_1$ and $z=z_2$ together with the union bound, we see that for $j=1,2$,
\[
\l| \widehat\Phi_{N\choose n}(z_j) - \Phi^{(n)}(z_j) \r| \leq \sqrt{\frac{s}{k}}
\]
on an event $\m E$ of probability $\geq 1 - 4e^{-2s}$.  
It follows that on $\m E$, $\med{W_J, J \in \m A_N^{(n)}}\in [z_2,z_1]$. 
The rest of the proof repeats the argument of section \ref{sec:proof1}. 

\subsection{Proof of Theorem \ref{th:mvar}.}
\label{sec:proof-mvar}

Set $F(z) := \sum_{j=1}^k \sum_{i=1}^m \rho_i(z_i - \bar\theta_{j,i})$. 
Then $\widehat\theta^{(k)} = \argmin_{z\in \mb R^m} F(z)$ by the definition. 
Since $F(z)$ is convex, the sufficient and necessary condition for $\widehat\theta^{(k)}$ to be its minimizer is 
that $0\in \partial F(\widehat\theta^{(k)})$, the subdifferential of $F$ at point $z$. 
It is easy to see that 
\[
\partial F(z) =  \left\{ u\in \mb R^m: \ 
\sum_{j=1}^k \rho'_{-,i}(z_i - \bar\theta_{j,i}) \leq u_i \leq  \sum_{j=1}^k \rho'_{+,i}(z_i - \bar\theta_{j,i}), \ i=1,\ldots,m\right\},
\]
where $\rho'_{+,i}(x) := \lim_{t\searrow 0} \frac{\rho_i(x+t) - \rho_i(x)}{t}$ and $\rho'_{-,i}(x) := \lim_{t\nearrow 0} \frac{\rho_i(x + t) - \rho_i(x)}{t}$ are the right and left derivative of $\rho_i$ at point $x$ respectively. 

Since the subdifferential is convex, it suffices to find points $z_{i,1}, z_{i,2}, \ i=1,\ldots,m$ such that for all $i$,
\begin{align}
\label{eq:subdif}
&
\sum_{j=1}^k \rho'_{-,i}(z_{i,1} - \bar\theta_{j,i})\leq 0, \\
&
\nonumber
\sum_{j=1}^k \rho'_{+,i}(z_{i,2} - \bar\theta_{j,i})\geq 0.
\end{align} 
This task has already been accomplished in the proof of Theorem \ref{th:M-est}: since $\rho_{+,i}, \ \rho_{-,i}, \ i=1,\ldots,m$ are nondecreasing functions, repeating the argument of section \ref{sec:proof-M} yields that, on an event of probability 
$\geq 1 - 4e^{-2s}$, inequalities \eqref{eq:subdif} hold with
\begin{align}
z_{i,1} & = \theta_{\ast,i} + 3\sigma_{n,i} e^{\l( C_{\rho_i}/\sigma_{n,i}\r)^2}  \l( \sqrt{\frac{s}{k}} + 2 g_m(n) \r), 
\\
\nonumber
z_{i,2} & = \theta_{\ast,i} - 3\sigma_{n,i} e^{\l( C_{\rho_i}/\sigma_{n,i}\r)^2}  \l( \sqrt{\frac{s}{k}} + 2 g_m(n) \r).
\end{align}
We have thus shown that for each $i=1,\ldots,m$, 
\[
\l| \widehat\theta_i^{(k)} - \theta_{\ast,i} \r| \leq 3 e^{\l( C_{\rho_i}/\sigma_{n,i}\r)^2} \cdot\sigma_{n,i}\l( \sqrt{\frac{s}{k}} + 2 g_m(n) \r)
\]
with probability $\geq 1 - 4e^{-2s}$. Applying the union bound over all $i$, we obtain the result. 

\subsection{Proof of Lemma \ref{cor:0}.}
\label{proof:cor0}

It is a simple numerical fact that whenever 
\[
\alpha_j \cdot \frac{1}{k}\sum_{j=1}^k \l( g_j(n_j) + \sqrt{\frac{s}{k}} \r)\leq 0.33,
\] 
$\zeta_j(n_j,s)/\sigma_{n_j}^{(j)}\leq 1$ (indeed, this follows since $\Phi(1)\simeq 0.8413>1/2+0.33$). 
Set $B(s):=\frac{1}{k}\sum_{j=1}^k \l( g_j(n_j) + \sqrt{\frac{s}{k}} \r)$ for brevity. 
Since $e^{-y^2/2}\geq 1 - \frac{y^2}{2}$, we have 
\begin{multline}
\label{eq:05}
\sqrt{2\pi} \alpha_j \cdot B(s) = \int_0^{\zeta_j(n_j,s)/\sigma_{n_j}^{(j)} } e^{-y^2/2} dy \\
\geq 
\frac{\zeta_j(n_j,s)}{\sigma_{n_j}^{(j)}} - \frac{1}{6}\l( \frac{\zeta_j(n_j,s)}{\sigma_{n_j}^{(j)}}\r)^3
\geq \frac{5}{6} \frac{\zeta_j(n_j,s)}{\sigma_{n_j}^{(j)}},
\end{multline} 
where the last inequality follows since $\zeta_j(n_j,s)/\sigma_{n_j}^{(j)}\leq 1$. 
Equation \eqref{eq:05} implies that 
$\frac{\zeta_j(n_j,s)}{\sigma_{n_j}^{(j)}} \leq \frac{6}{5}\alpha_j \sqrt{2\pi} B(s)$. 
Proceeding again as in \eqref{eq:05}, we see that 
\begin{multline*}
\sqrt{2\pi} \alpha_j \, B(s) \geq 
\frac{\zeta_j(n_j,s)}{\sigma_{n_j}^{(j)}} - \frac{1}{6}\l( \frac{\zeta_j(n_j,s)}{\sigma_{n_j}^{(j)}}\r)^3
\\
\geq 
\frac{\zeta_j(n_j,s)}{\sigma_{n_j}^{(j)}} - \frac{12\pi}{25}\alpha_j^2 \l( B(s)\r)^2 \frac{\zeta_j(n_j,s) }{\sigma_{n_j}^{(j)}}
\\
\geq \frac{\zeta_j(n_j,s)}{\sigma_{n_j}^{(j)}}\l( 1 - 1.51 \, \alpha_j^2\l(B(s) \r)^2 \r),
\end{multline*}
hence $\frac{\zeta_j(n_j,s)}{\sigma_{n_j}^{(j)}} \leq \frac{\sqrt{2\pi}}{1 - 1.51 \, \alpha_j^2\l( B(s) \r)^2 }\, \alpha_j B(s).$ 
The claim follows since $\alpha_j  B(s)\leq 0.33$ for all $j$ by assumption, and $\sigma_{n_j}^{(j)} \alpha_j \equiv H_k$. 

\section*{Acknowledgements}

Authors would like to thank Anatoli Juditsky for many insightful comments and suggestions.

\bibliographystyle{rss}
\bibliography{minskerstrawnbib,bibliography,bibliography2}

\appendix 

\section{Central limit theorem for the non-i.i.d. data.}
\label{appendix:clt}

We present an extension of Theorem \ref{th:clt} to non-i.i.d. data for the estimator $\widehat\theta^{(k)}=\med{\bar\theta_1,\ldots,\bar\theta_k}$ that holds under the following assumptions:
\begin{enumerate}
\item $X_1,\ldots,X_N$ are independent, $\card(G_j) = n_j$, and $\sum_{j=1}^k n_j = k$;
\item Assumption \ref{ass:1} is satisfied with some $\{\sigma_n^{(j)}\}_{n\geq 1}$ and $g_j(n)$, $j=1,\ldots,k$;
\item $k\to\infty$ and $\max_{j=1,\ldots,k}\sqrt{k}\cdot g_j(n_j)\to 0$ as $N\to \infty$;
\item $\max_{j\leq k} \frac{H_k}{\sigma_{n_j}^{(j)}\sqrt{k}}\xrightarrow{N\to\infty} 0$, where $H_k:=\l( \frac{1}{k}\sum_{j=1}^k \frac{1}{\sigma_{n_j}^{(j)}} \r)^{-1}$ is the harmonic mean of $\sigma_{n_j}^{(j)}$'s.  
\end{enumerate}
\begin{theorem}
\label{th:clt2}
Under assumptions (a)-(e) above, 
\[
\sqrt{k}\,\frac{\widehat\theta^{(k)} - \theta_\ast}{H_k}\xrightarrow{d} N\l(0,\frac{\pi}{2}\r).
\]
\end{theorem}
\begin{proof}
Define $d_-(x) := I\l\{ x >0 \r\} - I\l\{ x\leq 0\r\}$, and $Y_{n_j,j} = d_- \l( \theta_\ast - \bar\theta_j + t \sqrt{\frac{\pi}{2}} \frac{H_k}{\sqrt k} \r)$. 
We will show that 
\begin{enumerate}
\item $\frac{1}{k}\sum_{j=1}^k \sqrt{k}\mb EY_{n_j,j}\to t$ as $N\to\infty$;
\item $\frac{1}{k}\sum_{j=1}^k \var(Y_{n_j,j}) \to 1$ as $N\to\infty$.
\end{enumerate}
To prove the first claim, first assume that $t\ne 0$ (for $t=0$ the argument follows the same line with simplifications), and observe that 
\begin{multline*}
\sqrt{k}\mb EY_{n_j,j} 
=\sqrt{k}
\l( \mb E d_-\l( \frac{\theta_\ast - \bar\theta_j}{\sigma_{n_j}^{(j)}} + t \sqrt{\frac{\pi}{2}} \frac{H_k}{\sigma_{n_j}^{(j)}\sqrt k} \r) - \mb E d_-\l( Z + t \sqrt{\frac{\pi}{2}} \frac{H_k}{\sigma_{n_j}^{(j)}\sqrt k} \r) \r) \\
 +  t \sqrt{\frac{\pi}{2}} \frac{H_k}{\sigma_{n_j}^{(j)}} \cdot \frac{1}{  t \sqrt{\frac{\pi}{2}} \frac{H_k}{\sigma_{n_j}^{(j)}\sqrt k}} \l( \mb E d_- \l( Z + t \sqrt{\frac{\pi}{2}} \frac{H_k}{\sigma_{n_j}^{(j)}\sqrt k} \r) - \underbrace{ \mb E d_-\l( Z \r)}_{=0} \r).
\end{multline*}
Moreover,  
\[
\l|\sqrt{k}
\l( \mb E d_-\l( \frac{\theta_\ast - \bar\theta_j}{\sigma_{n_j}^{(j)}} + t \sqrt{\frac{\pi}{2}} \frac{H_k}{\sigma_{n_j}^{(j)}\sqrt k} \r) - \mb E d_-\l( Z + t \sqrt{\frac{\pi}{2}} \frac{H_k}{\sigma_{n_j}^{(j)}\sqrt k} \r) \r)\r| 
\leq 2 g_j(n_j),
\]
while under assumption (d),
\[
\frac{1}{  t \sqrt{\frac{\pi}{2}} \frac{H_k}{\sigma_{n_j}^{(j)}\sqrt k}} \l( \mb E d_- \l( Z + t \sqrt{\frac{\pi}{2}} \frac{H_k}{\sigma_{n_j}^{(j)}\sqrt k} \r) - \underbrace{ \mb E d_-\l( Z \r)}_{=0} \r) \to \frac{2}{\sqrt{2\pi}} \text{ as } N\to\infty.
\]
\end{proof}
It then follows from assumption (c) that 
\[
\l| \frac{1}{k}\sum_{j=1}^k \sqrt{k}\mb EY_{n_j,j} -  t \, \underbrace{H_k \frac{1}{k}\sum_{j=1}^k \frac{1}{\sigma_{n_j}^{(j)}} }_{=1}\r| \to 0 \text{ as } N\to\infty.
\]
Claim (b) follows since $\mb E \l(Y_{n_j,j} \r)^2 = 1$ and $\max_{j\leq k}\mb EY_{n_j,j}\to 0$ under assumption (d). 

The rest of the argument repeats the proof of Theorem \ref{th:clt} for $\rho(x)=|x|$.

\section{Supplementary results.}
\label{sec:supplement}

\begin{lemma}
\label{lemma:normal}
Let $\m A\subset \mb R$ be symmetric, meaning that $\m A=-\m A$, and let $Z\sim N(0,1)$. Then for all $x\in \mb R$,
\[
\pr{Z\in A - x}\geq e^{-x^2/2}\pr{Z\in A}.
\]
\end{lemma}
\begin{proof}
Observe that
\begin{align*}
\pr{Z\in A} & = \int_\mb R I\{z\in A\} \frac{1}{\sqrt{2\pi}}e^{-z^2/2} dz = e^{x^2/2}\int_\mb R I\{z\in A\} e^{-xz/2}e^{xz/2}\frac{1}{\sqrt{2\pi}}e^{-z^2/2} e^{-x^2/2} dz 
\\
& \leq e^{x^2/2} \sqrt{ \int_\mb R I\{z\in A\} \frac{1}{\sqrt{2\pi}}e^{-(z-x)^2/2}dz}\sqrt{ \int_\mb R I\{z\in A\} \frac{1}{\sqrt{2\pi}}e^{-(z+x)^2/2}dz} 
\\
& =  e^{x^2/2} \int_\mb R I\{z\in A\} \frac{1}{\sqrt{2\pi}}e^{-(z-x)^2/2}dz = e^{x^2/2}\,\pr{Z\in A - x},
\end{align*}
and the claim follows.
\end{proof}

\begin{lemma}
\label{lemma:tanh}
Inequality $\tanh(x)\geq x \l( \frac{1+x}{1+x+x^2} \r)$ holds for all $x\geq 0$. 
Moreover, if $\tanh(x)\leq 1/2$ and $x\geq 0$, then $\tanh(x)\geq 0.83x$.
\end{lemma}
\begin{proof}
Since $e^x\geq 1+x+\frac{x^2}{2}$ for all $x\geq 0$,
\begin{align*}
&
\tanh(x)=1 - \frac{2}{1+e^{2x}} \geq 1 - \frac{1}{1+x+x^2} = x \l( \frac{1+x}{1+x+x^2} \r).
\end{align*}
Note that $f(x)=\frac{1+x}{1+x+x^2}$ is decreasing on $[0,\infty)$. 
Whenever $\tanh(x)\leq 1/2$, $x\leq \frac{\log 3}{2}\leq 0.55$, hence 
$\tanh(x)\geq 0.83x$.
\end{proof}

\section{Results for the spatial median with respect to the $\|\cdot\|_2$ norm.}
\label{sec:l2-median}

In this section, we discuss estimation of the multivariate parameter $\theta_\ast \in \mb R^m$ based on the $L_2$-median. Let $X_1,\ldots,X_N\in \mb R^d$ be i.i.d. copies of $X$ randomly partitioned into disjoint groups $G_1,\ldots,G_k$ of cardinality $n\geq\lfloor N/k \rfloor$ each, and let $\bar\theta_j:=\bar\theta_j(G_j)\in \mb R^m, \ 1\leq j\leq k$ be a sequence of i.i.d. estimators of $\theta_\ast$. We define
\begin{align}
\label{eq:median}
\widehat\theta^{(k)}=\medg{\bar\theta_1,\ldots,\bar\theta_k}:=\argmin_{z\in \mb R^m}\sum_{j=1}^k \l\| z - \bar \theta_j \r\|_2
\end{align}
be the $L_2$ median of $\bar\theta_1,\ldots,\bar\theta_k$. 

Let $Z\in \mb R^m$ have multivariate normal distribution $N(0,\Sigma)$, and define $\Phi_\Sigma(A):=\pr{Z\in A}$ for a Borel measurable set $A\subseteq \mb R^m$. 
Moreover, define $\m S$ to be the set of closed cones, 
\begin{align}
\label{eq:cones}
&
\m S_m = \l\{ C_u(t;b) = \l\{ x\in \mb R^m: \dotp{x-b}{u}\geq t\|x-b\|_2 \r\}, \ \|u\|_2=1 \ ,b\in \mb R^m, \ 0\leq t\leq 1 \r\}.
\end{align}
We will assume that $\bar\theta_1$ is ``asymptotically normal on cones'':
\begin{assumption}
\label{ass:2}
There exists a sequence $\{\sigma_n\}_{n\in \mb N}\subset \mb R_+$ and a positive-definite matrix $\Sigma$ such that $\l\| \Sigma\r\| \leq 1$ and
\[
g_{\m S_m}(n):=\sup_{S\in \m S_m}\l| \pr{ \frac{1}{\sigma_n}\l(\bar\theta_1- \theta_\ast \r)\in S} - \Phi_\Sigma( S ) \r| \to 0 \text{ as }n\to\infty.
\]
\end{assumption}

\begin{theorem}
\label{th:main2}
Let assumption \ref{ass:2} be satisfied. Then with probability $\geq 1 - e^{-2s}$,
\begin{align}
\label{eq:mvar}
\tanh \l( \frac{1}{\sigma_n}\l\| \widehat \theta^{(k)} - \theta_\ast \r\|_2 \r) \leq 
26.8 \l\| \Sigma^{-1/2}\r\|
\l( \frac{C_1(m)}{\sqrt k} + C_2(m)\l( \sqrt{\frac{s}{4k}} + g_{\m S_m}(n) \r) \r),
\end{align}
where 
\[
C_1(m)=6\sqrt{\log 4e^{5/2}}(m+4)\sqrt{m+2\sqrt{(m-1)\ln 4}}
\] 
and $C_2(m)=\sqrt{m+2\sqrt{(m-1)\ln 4}}$.
\end{theorem}

\begin{remark}
\label{remark:2}
It follows from Lemma \ref{lemma:tanh} that whenever the right-hand side of the inequality \eqref{eq:mvar} is bounded by $1/2$, 
$\tanh \l( \frac{1}{\sigma_n}\l\| \widehat \theta^{(k)} - \theta_\ast \r\|_2 \r)\geq  \frac{0.83}{\sigma_n}\l\| \widehat \theta^{(k)} - \theta_\ast \r\|_2$, which leads to a more explicit bound for $\l\| \widehat \theta^{(k)} - \theta_\ast \r\|_2$.
\end{remark}

As an example, we consider the problem of the multivariate mean estimation. 
Recall that the condition number $\mathrm{cond}(A)$ of a non-singular matrix $A$ is defined as $\mathrm{cond}(A)=\|A\|\,\|A^{-1}\|$.  
\begin{corollary}
\label{corollary:med-of-means2}
Let $X_1,\ldots,X_N$ be a sequence of i.i.d. copies of a random vector $X\in\mb R^d$ such that $\mb EX = \theta_\ast$, 
$\mb E\l[(X-\theta_\ast)(X-\theta_\ast)^T\r]=\widetilde \Sigma$, and $\mb E\| X-\theta_\ast \|_2^3<\infty$. 
Define 
\[
\hat\theta^{(k)}=\medg{\bar\theta_1,\ldots,\bar\theta_k}.
\]
Assume that $s>0$ and $k\leq N/2$ are such that
\[
\mathrm{cond}(\widetilde \Sigma^{1/2})
\l( \frac{C_1(d)}{\sqrt k} + C_2(d)\l( \sqrt{\frac{s}{4k}} + \frac{400 d^{1/4} \mb E \l\| \widetilde\Sigma^{-1/2}(X-\theta_\ast) \r\|_2^3 }{\sqrt{n}} \r) \r)\leq 0.037.
\] 
Then
\begin{align*}
\Big\| \widehat \theta^{(k)} - \theta_\ast \Big\|_2 \leq & 32.4 \| \widetilde \Sigma^{1/2} \| \, \mathrm{cond}(\widetilde \Sigma^{1/2})\\
&
\times
\l( \frac{C_1(d)}{\sqrt {k n}} + C_2(d)\l( \sqrt{\frac{s}{4kn}} + \frac{400 d^{1/4} \mb E \l\| \widetilde\Sigma^{-1/2}(X-\theta_\ast) \r\|_2^3 }{n} \r) \r)
\end{align*}
with probability $\geq 1 - e^{-2s}$, where $C_1(d)$ and $C_2(d)$ are the same as in Theorem \ref{th:main2}.
\end{corollary}
\begin{proof}
It follows from the multivariate Berry-Esseen bound (fact \ref{th:MBE}) that assumption \ref{ass:2} is satisfied with 
$\sigma_n=\sqrt{\frac{\| \widetilde\Sigma\|}{n}}$, $\Sigma=\frac{\widetilde \Sigma}{\| \widetilde \Sigma \|}$ and 
$g_{\m S_d}(n)=\frac{400 d^{1/4} \mb E \l\| \widetilde\Sigma^{-1/2}X \r\|_2^3 }{\sqrt{n}}$. 
Noting that $\| \Sigma^{-1/2} \|=\| \widetilde \Sigma^{1/2} \| \, \| \widetilde \Sigma^{-1/2} \|=\mathrm{cond}(\widetilde \Sigma^{1/2})$, it is easy to deduce the bound from \eqref{eq:mvar} and remark \ref{remark:2}.
\end{proof}
\begin{remark}
\label{remark:3}
Note that, similarly to the case $d=1$, whenever $k\lesssim\sqrt{N}$ (hence, $n\gtrsim \sqrt{N}$), the bound of Corollary \ref{corollary:med-of-means2} is of order $N^{-1/2}$ with respect to the sample size $N$. 
However, dependence of the bound on the dimension factor $d$ is suboptimal. 
\end{remark}

\subsection{Overview of numerical algorithms.}
\label{sec:numerical} 

Letting $x_1,\ldots,x_k\in\mathbb{R}^d$, $F(z):=\sum\limits_{j=1}^k \|z-x_j\|$ is convex and it achieves its minimum at a unique point (unless 
$\{x_1,\ldots,x_k\}$ are on the same line \citep{kemperman1987median}) that belongs to the convex hull of $x_1,\ldots,x_k$. 

The classical algorithm that approximates $\argmin_{z\in \mb H} F(z)$ is the famous \textit{Weiszfeld's algorithm} \citep{Weiszfeld1936Sur-un-probleme00}: starting from some $z_0$ in the affine hull of $\{x_1,\ldots,x_k\}$, iterate
\begin{align}
\label{weizsfeld}
z_{m+1}=\sum_{j=1}^k \alpha^{(j)}_{m+1}\, x_j,
\end{align}
where $\alpha^{(j)}_{m+1}=\frac{\|x_j-z_m\|_\mb H^{-1}}{\sum\limits_{j=1}^k \|x_j-z_m\|_\mb H^{-1}}$. 
H. W. Kuhn \citep{kuhn1973note} showed that Weiszfeld's algorithm converges to the geometric median for all but countably many initial points. 
It is easy to check that (\ref{weizsfeld}) is a gradient descent scheme: indeed, it is equivalent to
\[
z_{m+1}=z_m-\beta_{m+1}g_{m+1},
\]
where $\beta_{m+1}=\frac{1}{\sum\limits_{j=1}^k\|x_j-z_m\|_\mb H^{-1}}$ and 
$g_{m+1}=\sum\limits_{j=1}^k \frac{z_m-x_j}{\|z_m-x_j\|_\mb H}$ is the gradient of $F$ (we assume here that $z_m\notin\{x_1,\ldots,x_k\}$). 

Various improvements and accelerated versions of Weiszfeld's algorithm have been proposed and analyzed.  \citet{ostresh1978convergence} provides a modified version of Weiszfeld's algorithm that converges to the geometric median under reasonable initialization conditions, but the rate of convergence is not specified. \citet{karkkainen2005computation} consider empirical behavior of several modifications of Weiszfeld's algorithm, and obtains convergence for an SOR method. 
\citet{vardi2000multivariate} demonstrate convergence of another modified Weiszfeld algorithm, but only provides empirical convergence rates. \citet{overton1983quadratically} provides an algorithm that exhibits quadratic convergence under some assumptions, but a quantitative rate is not expressed. \citet{cardot2013efficient} develops an online stochastic descent algorithms and provides an asymptotic convergence rate. Quantitative error bounds are not available for any of the algorithms discussed so far.

Literature from computer science considers the computational complexity of algorithms for computing $\widetilde{\theta}^{(k)}$ such that $F(\widetilde{\theta}^{(k)})$ is close to the minimum value $F(\widehat{\theta}^{(k)})$. A thorough comparison of such results is provided by \citet{cohen2016geometric}. 
The results from this work are fully quantitative, but they need to be adapted to our setting. In our statistical estimation setting, we are using $\widehat{\theta}^{(k)}$ to estimate the true parameter $\theta^\ast$, so we want bounds on the proximity $\Vert \widetilde{\theta}^{(k)}-\widehat{\theta}^{(k)}\Vert$ instead of bounds on $F(\widetilde{\theta}^{(k)})$. The following theorem (proven in Section \ref{proof:lowlip}) provides a ``local lower bound."

\begin{theorem}
\label{thm:lowlip}
Suppose $\{x_i\}_{i=1}^k$, let $\overline{x}=\frac{1}{k}\sum_{i=1}^kx_i$, set $m_t=\frac{1}{k}\sum_{i=1}^k\| x_i-\overline{x}\|^t$ for $t=1, 2, 3$, and assume that the empirical covariance matrix $\widehat{\Sigma}=\frac{1}{k}\sum_{i=1}^k (x_i-\overline{x})(x_i-\overline{x})^T$ satisfies
\[
a:=\frac{1}{k}\sum_{j=2}^d\lambda_j(\widehat{\Sigma})>0
\]
where $\lambda_j(\widehat{\Sigma})$ are the eigenvalues of $\widehat{\Sigma}$ listed with multiplicity and in non-increasing order. Then, for all $\theta\in\mb{R}^d$, 
\[
\frac{1}{k}(F(\theta)-F(\widehat{\theta}^{(k)}))\geq\frac{1}{2} \frac{a\|\theta-\widehat{\theta}^{(k)}\|^2}{b^2(\|\theta-\widehat{\theta}^{(k)}\|+b)}
\]
where 
\[
b = \frac{20m_1^3+6m_1m_2+m_3}{a}.
\]
\end{theorem}

Theorem \ref{thm:lowlip} allows us to infer proximity bounds from all the computer science literature that discusses value bounds. Moreover, this bound is asymptotically stable in the i.i.d. sampling setting assuming the existence of three moments. For small $\|\theta-\widehat{\theta}^{(k)}\|$, the lower bound is approximately quadratic, and hence the proximity bound behaves like $\sqrt{F(\theta)-F(\widehat{\theta}^{(k)})}$. On the other hand, this local lower bound fits in well with the theory of Restarted Gradient Descent \citep{yang2015rsg}. 

\section{Proofs of results in Appendix \ref{sec:l2-median}.}
\label{sec:proofs-l2}

\subsection{Technical background.}

Everywhere below, $\Phi_\Sigma$ stands for the distribution of the normal vector with mean 0 and covariance matrix $\Sigma$. 
The following multivariate version of the Berry-Esseen Theorem for convex sets has been established by \citet{bentkus2003dependence}. 
\begin{fact}[Multivariate Berry-Esseen bound]
\label{th:MBE}
Assume that $Y_1,\ldots,Y_n$ is a sequence of i.i.d. copies of a random vector $Y\in \mb R^d$ with mean $\mu$, covariance matrix $\Sigma\succ 0$ and such that $\mb E\|Y\|_2^3<\infty$. 
Let $Z$ have normal distribution $N(0,\Sigma)$, and $\m A$ be the class of all convex subsets of $\mb R^d$. 
Then 
\[
\sup_{A\in \m A} \l| \pr{\sqrt{n}(\bar Y - \mu)\in A } - \Phi_\Sigma(A) \r| \leq \frac{400 d^{1/4} \mb E \l\| \Sigma^{-1/2}Y \r\|_2^3 }{\sqrt{n}},
\]
where $\bar Y = \frac{1}{n}\sum_{j=1}^n Y_j$. 
\end{fact}

Given a metric space $(T,\rho)$, the covering number $N(T,\rho,\eps)$ is defined as the smallest $N\in \mb N$ such that there exists a subset 
$F\subseteq T$ of cardinality $N$ with the property that for all $z\in T$, $\rho(z,F)\leq \eps$. 
When metric $\rho$ is clear from the context, we will simply write $N(T,\eps)$. 

Let $\l\{ Y(t), \ t\in T\r\}$ be a stochastic process indexed by $T$. 
We will say that it has sub-Gaussian increments with respect to metric $\rho$ if for all $t_1,t_2\in \mb T$ and $s>0$,
\[
\pr{ | Y_{t_1} - Y_{t_2} | \geq s \rho(t_1,t_2) } \leq 2e^{-s^2/2}.
\]
\begin{fact}[Dudley's entropy bound]
\label{th:DUD}
Let $\{Y(t), \ t\in T \}$ be a centered stochastic process with sub-Gaussian increments. Then the following inequality holds:
\begin{align}
\label{eq:dudley}
\mb E \sup_{t\in T} Y(t)\leq 12\int\limits_{0}^{D(T)} \sqrt{\log N(T,\rho,\eps)}d\eps,
\end{align}
where $D(T)$ is the diameter of the space $T$ with respect to $\rho$.
\end{fact}
\begin{proof}
See \citep{talagrand2005generic}.
\end{proof}

Finally, we recall two useful facts related to Vapnik-Chervonenkis (VC) combinatorics \citep[see][for the definition of VC dimension and related theory]{wellner1}. 
Let $\m F$ be a finite-dimensional vector space of real functions on $S$. 
\begin{fact}
\label{th:VC1}
Let $\m C=\l\{ \{ f\geq 0\}: \ f\in \m F \r\}$ and $\m C_+=\l\{ \{ f > 0\}: \ f\in \m F \r\}$
Then 
\[
\mathrm{VC}(\m C)=\mathrm{VC}(\m C_+)=\dim(\m F).
\]
\end{fact}
\begin{proof}
See Proposition 3.6.6 in \citep{gine2015mathematical}. 
\end{proof}

\begin{fact}
\label{th:VC2}
Let $\m C$ be a class of sets of VC-dimension $V$. 
Then, for any probability measure $Q$, 
\begin{align}
\label{eq:cov-number}
&
N(\m C, L_2(Q),\eps) \leq e (V +1) (4e)^{V} \l( \frac{1}{\eps^2} \r)^{V}
\end{align}
for all $0<\eps\leq 1$; 
\end{fact}
\begin{proof}
This bound follows from results of R. Dudley \citep{dudley1978central} and D. Haussler \citep{haussler1995sphere}. 
The bound with explicit constants as stated above is given in \citep{pollardasymptopia}.
\end{proof}

\subsection{Proof of Theorem \ref{th:main2}.}
\label{sec:proof2}

By the definition of the geometric median, 
\[
\widehat \theta^{(k)} = \argmin_{z\in \mb R^m}\sum_{j=1}^k \| z - \bar \theta_j\|_2,
\]
hence 
\begin{align}
\label{eq:01}
&
\frac{1}{\sigma_n}\l( \widehat \theta^{(k)} - \theta_\ast \r) = \argmin_{z\in \mb R^m} \sum_{j=1}^k \l\| z - \frac{1}{\sigma_n} \l(\bar \theta_j - \theta_\ast \r)  \r\|_2. 
\end{align}
Set $F_k(z) := \sum_{j=1}^k \l\| z - \frac{1}{\sigma_n} \l(\bar \theta_j - \theta_\ast \r)  \r\|_2$. 
Then \eqref{eq:01} is equivalent to 
\[
\widehat \mu^{(k)}:=\frac{1}{\sigma_n}\l( \widehat \theta^{(k)} - \theta_\ast \r) = \argmin_{z\in \mb R^m} F(z).
\]  
Denote by $\Phi^{(n)}$ the distribution of $\frac{1}{\sigma_n} \l(\bar \theta_1 - \mu \r)$, and by $\Phi^{(n)}_{k}$ - the empirical distribution corresponding to the sample 
\[
W_1=\frac{1}{\sigma_n}(\bar \theta_1 - \mu), \ldots, W_k=\frac{1}{\sigma_n}(\bar \theta_k - \mu).
\] 
Let $DF_k(\widehat \mu^{(k)};u):=\lim_{t\searrow 0} \frac{F_k(\widehat \mu^{(k)}+tu) - F_k(\widehat \mu^{(k)})}{t}$ be the directional derivative of $F_k$ at point $\widehat \mu^{(k)}$ in direction $u$. 
Clearly, $DF_k(\widehat \mu^{(k)};u)\geq 0$ for any $u$ such that $\|u\|_2=1$. 
On the other hand, it is easy to check that $DF_k(\widehat \mu^{(k)};u) = \Phi^{(n)}_{k} f_{u,\widehat \mu^{(k)}}$,
where 
\[
f_{u,b}(x) = \begin{cases}
\dotp{\frac{x - b }{\| x - b \|_2 }}{u}, & x\ne b, \\
1, & x = b.
\end{cases}
\]
Let $\m S_m$ be the set of closed cones defined in \eqref{eq:cones}, and note that for any unit vector $u\in \mb R^m$ and $t\in[0,1]$, 
\begin{align}
\label{eq:cone}
&
\l\{ x\in \mb R^m: \ f_{u,\widehat \mu^{(k)}}(x)\geq t \r\} = C_u(t;\widehat \mu^{(k)}).
\end{align} 
Next, observe that 
\begin{align}
\label{eq:main-0}
&
0\leq DF(\widehat \mu^{(k)};u) = ( \Phi^{(n)}_{k} - \Phi^{(n)}) f_{u,\widehat \mu^{(k)}} + (\Phi^{(n)} - \Phi_{\Sigma}) f_{u,\widehat \mu^{(k)}} + \Phi_{\Sigma} \, f_{u,\widehat \mu^{(k)}}.
\end{align}
We will assume that $u$ is \emph{chosen} such that $\Phi_{\Sigma} \, f_{u,\widehat \mu^{(k)}} \leq 0$ (if not, simply replace $u$ by $-u$). 
Then \eqref{eq:main-0} implies that
\begin{align}
\label{eq:15}
&
\Phi_{\Sigma} \, f_{-u,\widehat \mu^{(k)}} \leq  \l| (\Phi^{(n)}_{k} - \Phi^{(n)}) f_{u,\widehat \mu^{(k)}}  \r| + 
\l| (\Phi^{(n)} - \Phi_{\Sigma}) \, f_{u,\widehat \mu^{(k)}} \r|.
\end{align}
It remains to estimate the left-hand side of inequality \eqref{eq:15} from below and its right-hand side from above. 
We start by finding an upper bound (proved in section \ref{proof:lemma0}) for $\l| (\Phi^{(n)} - \Phi_{\Sigma}) \, f_{u,\widehat \mu^{(k)}} \r|$.
\begin{lemma}
\label{lemma:0}
The following bound holds:
\[
\l| (\Phi^{(n)} - \Phi_{\Sigma}) \, f_{u,\widehat \mu^{(k)}} \r| \leq 2g_{\m S_m}(n),
\]
where $g_{\m S_m}(n)$ was defined in assumption \ref{ass:2}.
\end{lemma}
The next Lemma (proved in section \ref{proof:lemma1}) provides an upper bound for 
$\l| (\Phi^{(n)}_{k} - \Phi^{(n)}) f_{u,\widehat \mu^{(k)}} \r|$.
\begin{lemma}
\label{lemma:1}
With probability $\geq 1 - e^{-2s}$,
\[
\l| (\Phi^{(n)}_{k} - \Phi^{(n)}) f_{u,\widehat \mu^{(k)}} \r| \leq \frac{12(m+4)}{\sqrt k}\sqrt{\log 4e^{5/2}}+ \sqrt{\frac{s}{k}}. 
\]
\end{lemma}

Finally, it remains to estimate $\Phi_{\Sigma} \, f_{-u,\widehat \mu^{(k)}}$ from below. 
The following inequality (proved in section \ref{proof:lemma2}) holds:
\begin{lemma}
\label{lemma:2}
Set $u = -\frac{\Sigma^{-1}\widehat \mu^{(k)}}{\|\Sigma^{-1}\widehat \mu^{(k)}\|_2}$. Then
\[
\Phi_\Sigma \, f_{-u,\widehat \mu^{(k)}} \geq \frac{0.15 }{2 \l\| \Sigma^{-1/2}\r\| \sqrt{m+2\sqrt{(m-1)\ln 4}}}\tanh \l( \l\| \widehat \mu^{(k)} \r\|_2 \r),
\]
where $\tanh(\cdot)$ is the hyperbolic tangent defined as $\tanh(x)=\frac{1-e^{-2x}}{1+e^{-2x}}$.
\end{lemma}
It therefore follows from Lemmas \ref{lemma:0}, \ref{lemma:1} and \ref{lemma:2} that with probability exceeding $1 - e^{-2s}$,
\[
\frac{0.15}{2 \l\| \Sigma^{-1/2}\r\| \sqrt{m+2\sqrt{(m-1)\ln 4}}}\tanh \l( \l\| \widehat \mu^{(k)} \r\|_2 \r) \leq 
\frac{12(m+4)}{\sqrt k}\sqrt{\log 4e^{5/2}}+ \sqrt{\frac{s}{k}} +
2 g_\m S(n),
\]	
which implies the bound of Theorem \ref{th:main2}.

\subsubsection{Proof of Lemma \ref{lemma:0}.}
\label{proof:lemma0}

Recall that for any non-negative function $f:\mb R^m\mapsto \mb R_+$ and a signed measure $Q$, 
\begin{align}
\label{eq:transform}
\int_{\mb R^m} f(x)dQ = \int_{0}^\infty Q\l( x: \,f(x)\geq t \r) dt.
\end{align} 
Hence 
\begin{align*}
&
(\Phi^{(n)} - \Phi_{ \Sigma}) \, f_{u,\widehat \mu^{(k)}} = 
(\Phi^{(n)} - \Phi_{ \Sigma}) \, \max\l( f_{u,\widehat \mu^{(k)}},0 \r) - 
(\Phi^{(n)} - \Phi_{ \Sigma}) \, \max\l( f_{-u,\widehat \mu^{(k)}},0\r), 
\end{align*}
where we used the identity $-f_{u,\widehat \mu^{(k)}}=f_{-u,\widehat \mu^{(k)}}$. 
Next, it follows from \eqref{eq:cone} that
\begin{align*}
\l|(\Phi^{(n)} - \Phi_{ \Sigma}) \, \max\l( f_{u,\widehat \mu^{(k)}},0 \r)\r| &= 
\l| \int_0^1 (\Phi^{(n)} - \Phi_{ \Sigma})\l( x: \, f_{u,\widehat \mu^{(k)}}\geq t \r)dt \r| \\
&\leq
\max_{0\leq t\leq 1}\l| (\Phi^{(n)} - \Phi_{ \Sigma}) C_u(t; \widehat \mu^{(k)}) \r| \leq g_{\m S_m}(n)
\end{align*}
by assumption \ref{ass:2}. It implies that
$
\l| (\Phi^{(n)} - \Phi_{ \Sigma}) \, f_{u,\widehat \mu^{(k)}} \r| \leq 2g_{\m S_m}(n),
$ 
as claimed. 

\subsubsection{Proof of Lemma \ref{lemma:1}.}
\label{proof:lemma1}

Using \eqref{eq:transform} and proceeding as in the proof of Lemma \ref{lemma:0}, we obtain that
\[
\l| (\Phi^{(n)}_{k} - \Phi^{(n)}) f_{u,\widehat \mu^{(k)}} \r| \leq 
\max_{0\leq t\leq 1} \l| (\Phi^{(n)}_{k} - \Phi^{(n)}) C_u(t; \widehat \mu^{(k)}) \r| 
\leq \sup_{A\in \m S_m} \l| \Phi^{(n)}_{k}(A) - \Phi^{(n)}(A) \r|. 
\]
It follows from the bounded difference inequality (fact \ref{th:BDI}) that for all $s>0$,
\[
\pr{ \sup_{A\in \m S_m}\l| \Phi^{(n)}_{k}(A) - \Phi^{(n)}(A) \r| -\mb E \sup_{A\in \m S_m}\l| \Phi^{(n)}_{k}(A) - \Phi^{(n)}(A) \r| 
\geq \sqrt{\frac{s}{k}} }\leq e^{-2s},
\]
hence it is enough to control $\mb E \sup_{A\in \m S_m}\l| \Phi^{(n)}_{k}(A) - \Phi^{(n)}(A) \r|$. 
To this end, we will estimate the covering numbers of the class of cones $\m S$ and use Dudley's integral bound (fact \ref{th:DUD}). 

Given a vector $\mf{x}\in \mb R^m$, let $x_1,\ldots,x_m$ be its coordinates with respect to the standard Euclidean basis. 
Note that 
\begin{align*}
&
\dotp{\mf{x} - b}{u}\geq t\| \mf{x} -b\|_2 \iff \dotp{\mf{x}-b}{u}\geq 0 \text { and } \dotp{\mf{x}-b}{u}^2 \geq t^2\|\mf{x} - b\|^2_2, 
\end{align*}
which is equivalent to 
$\sum_{i,j=1}^m \alpha_i \alpha_{i,j} x_i x_j +\sum_{j=1}^m \beta_j x_j + \gamma \geq 0$ and $\dotp{x-b}{u}\geq 0$, 
where $\alpha_{i,j}, \ \beta_j, \ i,j=1,\ldots, m$, and $\gamma$ are functions of $t, \ b_j$ and $u_j$, $j=1,\ldots,m$. 
In particular, every element of $A\in \m S_m$ is the intersection of a half-space $\l\{ \mf{x}: \ \dotp{\mf{x}-b}{u}\geq 0\r\}$ and a set 
$\l\{\mf{x}: \ f(\mf{x})\geq 0 \r\}$, where $f$ is a polynomial of degree $2$ in $m$ variables. 
The dimension of the space $V_{2,m}$ of polynomials of degree at most $2$ is $\dim(V_{2,m}) = {m+2\choose 2} $, hence 
the Vapnik-Chernonenkis dimension of the collection of sets 
$\m S_{V_{2,m}} = \Big\{ \l\{ x: f(x)\geq 0\r\}, \ f\in V_{2,m} \Big\}$ is $\tilde m:={m+2\choose 2}$ by fact \ref{th:VC1}. 
It follows from fact \ref{th:VC2} that for any probability measure $Q$, 
\begin{align}
\label{eq:cov-number}
&
N(\m S_{V_{2,m}}, L_2(Q),\eps) \leq e (\tilde m +1) (4e)^{\tilde m} \l( \frac{1}{\eps^2} \r)^{\tilde m}
\end{align}
for all $0<\eps\leq 1$.
It is also well known that (and can be deduced from the similar reasoning) that the VC-dimension of a collection $\m S_L$ of halfspaces of $\mb R^m$ is $m+1$, hence 
\[
N(\m S_L, L_2(Q),\eps) \leq e (m + 2) (4e)^{m+1} \l( \frac{1}{\eps^2} \r)^{m+1}.
\] 
Given two collections of sets $\m C_1, \ \m C_2$, let $A^{(1)}_1,\ldots,A^{(1)}_{N(\m C_1,L_2(Q),\eps)}$ and $A^{(2)}_1,\ldots, A^{(2)}_{N(\m C_2,L_2(Q),\eps)}$ be the $L_2(Q)$ $\eps$ - nets of smallest cardinality for the classes of functions $\l\{ I_{A}: \ A\in \m C_1\r\}$ and $\l\{ I_A: \ A\in \m C_2 \r\}$ respectively. 
Let $A'\in \m C_1, \ A''\in \m C_2$, and assume without loss of generality that $\|A' - A^{(1)}_1\|_{L_2(Q)}\leq \eps$ and $\|A'' - A^{(2)}_1\|_{L_2(Q)}\leq \eps$. Then
\begin{align*}
&
\l\| I_{A'} I_{A''} - I_{A^{(1)}_1} I_{A^{(2)}_2}\r\|_{L_2(Q)}\leq 2\eps, 
\end{align*}
which implies that the covering number of the class $\m D = \l\{  I_{A_1} I_{A_2}, \ A_1\in \m C_1, \ A_2\in \m C_2\r\}$ 
corresponding to intersections of elements of $\m C_1$ and $\m C_2$ satisfies 
\[
N\l( \m D,L_2(Q),\eps \r)\leq N(\m C_1,L_2(Q),\eps/2) N(\m C_2,L_2(Q),\eps/2).
\]
In particular, the metric entropy of the class of cones $\m S_m$ can be bounded as
\[
\log N\l( \m S_m,L_2(Q),\eps \r)\leq 2\l( {m+2\choose 2} + m + 1 \r)\log \frac{4e^{3/2}}{\eps}
\]
uniformly over all probability measures $Q$, hence fact \ref{th:DUD} implies that
\begin{align*}
\mb E \sup_{A\in \m S_m}&\l| \Phi^{(n)}_k(A) - \Phi^{(n)} (A) \r| \leq \frac{12}{\sqrt k}\mb E\l[ \int_0^1 \sqrt{\log N\l( \m S_m,L_2\l(\Phi^{(n)}_k\r),\eps \r)}d\eps \r]  \\
&
\leq \frac{12}{\sqrt k}\mb E\l[ \sqrt{\int_0^1 \log N\l( \m S_m,L_2\l(\Phi^{(n)}_k\r),\eps \r)d\eps}\r]  \leq \frac{12(m+4)}{\sqrt k}\sqrt{\log 4e^{5/2}}.
\end{align*}

\subsubsection{Proof of Lemma \ref{lemma:2}.}
\label{proof:lemma2}

Making the change of variables $x=\Sigma^{1/2}z$, we obtain
\begin{align*}
\int_{\mb R^m} \dotp{\frac{x-\widehat \mu^{(k)}}{\| x-\widehat \mu^{(k)}\|_2}}{u}d\Phi_{\Sigma}(x) & = 
\int_{\mb R^m}\dotp{\frac{ \Sigma^{1/2}( z -  \Sigma^{-1/2}\widehat \mu^{(k)}) }{\l\|  \Sigma^{1/2}( z -  \Sigma^{-1/2} \widehat \mu^{(k)}) \r\|_2}}{u}d\Phi(z)
\\
& 
\geq
\l\|  \Sigma^{1/2} u \r\|_2 \int_{\mb R^m}\dotp{\frac{z - \Sigma^{-1/2}\widehat \mu^{(k)} }{\l\| z - \Sigma^{-1/2} \widehat \mu^{(k)} \r\|_2}}{ \tilde u}d\Phi(z),
\end{align*}
where $\tilde u=\frac{ \Sigma^{1/2} u}{\| \Sigma^{1/2} u \|_2}$. 
Let $\kappa:=\l\| \Sigma^{-1/2}\widehat \mu^{(k)} \r\|_2$, and note that $\kappa\geq \l\| \widehat \mu^{(k)} \r\|_2$ since $\| \Sigma \|\leq 1$ by assumption. 
Let $V$ be any orthogonal transformation that maps 
$\Sigma^{-1/2}\widehat \mu^{(k)}$ to $\kappa e_1$ (here, $e_1,\ldots,e_m$ is the standard Euclidean basis of $\mb R^m$). 
Then, letting $y=V (z - \Sigma^{-1/2}\widehat \mu^{(k)})$, we observe that 
\[
\int_{\mb R^m} \dotp{\frac{x-\widehat \mu^{(k)}}{\| x-\widehat \mu^{(k)}\|_2}}{u}d\Phi_{\Sigma}(x) \geq
\l\| \Sigma^{1/2} u \r\|_2 \int_{\mb R^m} \dotp{\frac{y}{\|y\|_2}}{V\tilde u}d\Phi( y +\kappa e_1).
\]
Setting $u = -\frac{\Sigma^{-1}\widehat \mu^{(k)}}{\| \Sigma^{-1}\widehat \mu^{(k)}\|_2}$, we obtain from the last inequality that 
\begin{align*}
\int_{\mb R^m} \dotp{\frac{x-\widehat \mu^{(k)}}{\| x-\widehat \mu^{(k)}\|_2}}{u}d\Phi_{\Sigma}(x) &\geq
\frac{1}{\l\| \Sigma^{-1/2}\r\|} 
\int_{\mb R^m} \dotp{\frac{y}{\|y\|_2}}{-e_1}d\Phi( y +\kappa e_1). 
\end{align*}
Set $y=(-t,z)$, where $t\in \mb R$ and $z\in \mb R^{m-1}$. 
We will also let $\phi_k$ denote the density (with respect to Lebesgue measure) of the standard normal distribution on $\mb R^k$. 
Then
\begin{align*}
&
\int_{\mb R^m} \dotp{\frac{y}{\|y\|_2}}{-e_1}d\Phi( y +\kappa e_1) = 
\int_{\R^{m-1}} \int_{-\infty}^\infty \frac{t}{\sqrt{t^2+\| z\|_2^2}} \phi_1(t-\kappa)\phi_{m-1}(z)\:dt\:dz.
\end{align*}

Setting $h(t,z) = t/\sqrt{t^2+\Vert z\Vert_2^2}$, we have that
\begin{align}
\label{eq:b05}
\nonumber
\int_{-\infty}^{\infty} h(t,z)\phi_1(t-\kappa)\:dt &
= \int_{-\infty}^0 h(t,z)\phi_1(t-\kappa)\:dt + \int_0^{\infty} h(t,z)\phi_1(t-\kappa)\:dt\\
\nonumber
&= \int_{\infty}^0 h(t,z)\phi_1(-t-\kappa)\:dt + \int_0^{\infty} h(t,z)\phi_1(t-\kappa)\:dt\\
\nonumber
&= \int_{\infty}^0 h(t,z)\phi_1(t+\kappa)\:dt + \int_0^{\infty} h(t,z)\phi_1(t-\kappa)\:dt\\
&= \int_0^\infty h(t,z)\left[\phi_1(t-\kappa)-\phi_1(t+\kappa)\right]\:dt.
\end{align}
Now, for any $t\geq 0$, 
\begin{align*}
\phi_1(t-\kappa)-\phi_1(t+\kappa) &= \frac{e^{-(t^2+\kappa^2)/2}}{\sqrt{2\pi}}\left(e^{t\kappa}-e^{-t\kappa}\right)\\
&= \frac{e^{-(t^2+\kappa^2)/2}}{\sqrt{2\pi}}\tanh(t\kappa)\left(e^{t\kappa}+e^{-t\kappa}\right)\\
&\geq \frac{e^{-(t^2+\kappa^2)/2}}{\sqrt{2\pi}}\tanh(t\kappa)e^{t\kappa}=\tanh(t\kappa) \phi_1(t-\kappa),
\end{align*}
hence 
\begin{align*}
\int_{\mb R^m} \dotp{\frac{y}{\|y\|_2}}{-e_1}d\Phi( y +\kappa e_1) & \geq 
\int_{\mb R^{m-1}}\int_0^\infty h(t,z) \tanh(t\kappa)\phi_1(t-\kappa)\phi_{m-1}(z)\:dt\:dz \\
& \geq
\int\limits_{\|z\|_2\leq R}\int_1^\infty h(t,z)\tanh(t\kappa)\phi_1(t-\kappa)\phi_{m-1}(z)\:dt\:dz \\
& \geq 
\frac{\tanh(\kappa)}{\sqrt{1+R^2}}\int\limits_{\|z\|_2\leq R} \phi_{m-1}(z)dz\int_1^\infty \phi_1(t-\kappa)dt \\
&\geq
\frac{0.15\tanh(\kappa)}{\sqrt{1+R^2}} \int\limits_{\|z\|_2\leq R}\phi_{m-1}(z)dz,
\end{align*}
where we have use the inequality $h(t,z)\geq (1+R^2)^{-1/2}$ whenever $\|z\|^2_2\leq R$ and $t\geq 1$, and 
$1 - \Phi(1) > 0.15$. 
Finally, a well-known bound states that if $Y$ has $\chi_{m-1}^2$ distribution, then for all $t>0$
\[
\pr{\frac{Y}{m-1} - 1>t }\leq e^{-(m-1)t^2/8}.
\]
For $R^2:=m - 1 + 2\sqrt{(m-1)\ln 4}$, it implies that 
\[
\int\limits_{\|z\|_2\leq R}\phi_{m-1}(z)dz = \pr{ Y \leq R^2} = 
\pr{\frac{Y}{m-1} - 1 \leq 2\sqrt{\frac{\log 4}{m-1}}}
\geq 1/2,
\] 
which concludes the proof.

\subsection{Proof of Theorem \ref{thm:lowlip}.}
\label{proof:lowlip}

To simplify notation in what follows, we let $z^\ast=\argmin_{z\in\mb{R}^d} F(z)$. We let $f_i(z) = \Vert z - x_i\Vert$ for all $i=1,\ldots, k$ and observe that a weak gradient of $f_i(z)$ is given by
\[
\nabla f_i(z) = \left\{\begin{array}{cl}
				0 & z=0\\
				\frac{1}{\Vert z - x_i\Vert}(z-x_i) & z\not=0
				\end{array}\right..
\] 
Hence, $\nabla F(z) = \sum_{i=1}^k \nabla f_i(z)$ is a weak gradient of $F$. 

Now, fix $z\in\mb{R}^d$ with $z\not=z^\ast$, let $r=\Vert z-z^\ast\Vert$, and set $u=\frac{1}{r}(z-z^\ast)$. The second fundamental theorem of calculus yields
\begin{align*}
F(z) - F(z^\ast) &= \int_0^r \nabla F(z^\ast + t u)^T u dt\\
&=\int_0^r \sum_{i=1}^k \frac{1}{\|z^\ast-x_i+ t u\|}(z^\ast-x_i+tu)^Tu dt\\
&=\int_0^r \sum_{i=1}^k \frac{1}{\|z^\ast-x_i+ t u\|}(z^\ast-x_i+tu)^Tu dt\\
&=\int_0^r \sum_{i=1}^k \frac{(z^\ast-x_i)^Tu + t}{\sqrt{\|z^\ast-x_i\|^2+2t(z^\ast-x_i)^Tu + t^2}}dt\\
&=\int_0^r \sum_{i=1}^k \frac{\gamma_i c_i + t}{\sqrt{(\gamma_ic_i + t)^2+\gamma_i^2(c_i^2-1)}}dt.
\end{align*}
In this last line, we have set $\gamma_i = \|z^\ast - x_i\|$ and $c_i= \frac{1}{\gamma_i}(z^\ast-x_i)^T u$. By Cauchy-Schwarz, we have that $c_i^2\leq 1$. If $c_i^2=1$, then 
\[
\frac{\gamma_i c_i + t}{\sqrt{(\gamma_i c_i +t)^2 + \gamma_i^2(1-c_i^2)}}= \text{sgn}(\gamma_i c_i+t)\geq c_i
\]
for all $t\geq 0$. If $c_i^2<1$, then we have that
\[
\frac{\gamma_i c_i + t}{\sqrt{(\gamma_i c_i +t)^2 + \gamma_i^2(1-c_i^2)}}=c_i+\int_0^t \frac{\gamma_i^2(1-c_i^2)}{\left[(\gamma_ic_i+s)^2+\gamma_i^2(1-c_i^2)\right]^{3/2}}ds.
\]
Note that $\sum_{i=1}^k c_i = \nabla F(z^\ast)^T u=0$ since $z^\ast$ is the minimizer. Consequently, we have
\begin{align*}
F(z) - F(z^\ast) &\geq \int_0^r \left(\sum_{i=1}^k c_i + \sum_{i:c_i^2<1} \int_0^t \frac{\gamma_i^2(1-c_i^2)}{\left[(\gamma_ic_i+s)^2+\gamma_i^2(1-c_i^2)\right]^{3/2}}ds\right)dt\\
&=\sum_{i:c_i^2<1} \int_0^r\int_0^t \frac{\gamma_i^2(1-c_i^2)}{\left[(\gamma_ic_i+s)^2+\gamma_i^2(1-c_i^2)\right]^{3/2}}ds\:dt\\
&=\sum_{i:c_i^2<1} \int_0^r\int_0^t \frac{1-c_i^2}{\gamma_i}\frac{1}{\left[(c_i+\frac{s}{\gamma_i})^2+(1-c_i^2)\right]^{3/2}}ds\:dt.
\end{align*}
Given that 
\[
\left(c_i+\frac{s}{\gamma_i}\right)^2+(1-c_i^2)=\frac{s^2}{\gamma_i^2}+2c_i\frac{s}{\gamma_i}+1\leq \frac{s^2}{\gamma_i^2}+2\frac{s}{\gamma_i}+1=\left(1+\frac{s}{\gamma_i}\right)^2,
\]
we obtain the lower bound
\begin{align*}
F(z) - F(z^\ast)&\geq\sum_{i:c_i^2<1} \int_0^r\int_0^t \frac{1-c_i^2}{\gamma_i}\frac{1}{\left[(\frac{s}{\gamma_i}+1)^2\right]^{3/2}}ds\:dt\\
&=\sum_{i:c_i^2<1} \int_0^r\int_0^t \frac{1-c_i^2}{\gamma_i}\frac{1}{(\frac{s}{\gamma_i}+1)^3}ds\:dt\\
&=\sum_{i:c_i^2<1} \int_0^r\int_0^t \frac{\gamma_i^2(1-c_i^2)}{(s+\gamma_i)^3}ds\:dt\\
&=\left(\sum_{j=1}^k \gamma_j^2(1-c_j^2)\right) \int_0^r\int_0^t \sum_{i=1}^k\frac{\gamma_i^2(1-c_i^2)}{\sum_{j=1}^k \gamma_j^2(1-c_j^2)}\frac{1}{(s+\gamma_i)^3}ds\:dt.
\end{align*}
Noting that the inverse cubic function is convex, Jensen's inequality and straightforward integration yields
\begin{align*}
F(z) - F(z^\ast) &\geq\left(\sum_{j=1}^k \gamma_j^2(1-c_j^2)\right)\int_0^r\int_0^t \frac{1}{\left(s+\frac{\sum_{i=1}^k \gamma_i^3(1-c_i^2)}{\sum_{j=1}^k \gamma_j^2(1-c_j^2)}\right)^3}ds\:dt\\
&=\frac{1}{2}\left(\sum_{j=1}^k \gamma_j^2(1-c_j^2)\right) \frac{r^2}{\left(\frac{\sum_{i=1}^k \gamma_i^3(1-c_i^2)}{\sum_{j=1}^k \gamma_j^2(1-c_j^2)}\right)^2\left(r+\frac{\sum_{i=1}^k \gamma_i^3(1-c_i^2)}{\sum_{j=1}^k \gamma_j^2(1-c_j^2)}\right)}.
\end{align*} 

We now observe that
\[
\sum_{i=1}^k \gamma_i^3(1-c_i^2)\leq \sum_{i=1}^k\|z^\ast - x_i\|^3\leq \sum_{i=1}^k \left(\|z^\ast - \overline{x}\|+\|\overline{x}-x_i\|\right)^3\leq  \sum_{i=1}^k \left(\frac{2}{k} F(\overline{x})+\|\overline{x}-x_i\|\right)^3
\]
and also that
\[
\sum_{i=1}^k \gamma_i^2(1-c_i^2) = \sum_{i=1}^k \|z^\ast-x_i\|^2-\left((z^\ast-x_i)^Tu\right)^2=\sum_{i=1}^k\sum_{j=2}^d u_j^T(z^\ast-x_i)(z^\ast-x_i)^T u_j
\]
where $\{u, u_2, \ldots, u_d\}$ is an orthonormal basis for $\mb{R}^d$. We further observe that
\begin{align*}
\sum_{i=1}^k (z^\ast-x_i)(z^\ast-x_i)^T &= \sum_{i=1}^k (z^\ast -\overline{x}+\overline{x}-x_i)(z^\ast-\overline{x}+\overline{x}-x_i)^T\\
&=k(z^\ast-\overline{x})(z^\ast-\overline{x})^T +\sum_{i=1}^k(x_i-\overline{x})(x_i-\overline{x})^T.
\end{align*}
The Courant-Fischer characterization of eigenvalues gives us
\[
\sum_{i=1}^k \gamma_i^2(1-c_i^2)\geq \sum_{j=2}^d u_j^T\left( \sum_{i=1}^k(x_i-\overline{x})(x_i-\overline{x})^T\right)u_j\geq k\sum_{j=2}^d\lambda_j(\widehat{\Sigma})
\]
where $\{\lambda_j(\widehat{\Sigma})\}_{j=1}^d$ are the eigenvalues of the empirical covariance matrix listed with multiplicity and in non-increasing order. We therefore have
\[
\frac{1}{k}(F(z)-F(z^\ast))\geq \frac{1}{2} \frac{\sum_{j=2}^d \lambda_j(\widehat{\Sigma}) r}{\left(\frac{\frac{1}{k}\sum_{i=1}^k \left(2 m_1+\|x_i-\overline{x}\|\right)^3}{\sum_{j=2}^d \lambda_j(\widehat{\Sigma})}\right)^2\left(r+\frac{\frac{1}{k}\sum_{i=1}^k \left(2 m_1+\|x_i-\overline{x}\|\right)^3}{\sum_{j=2}^d \lambda_j(\widehat{\Sigma})}\right)},
\]
and the result follows.

\end{document}